\documentclass[11pt]{article}

\usepackage{epsfig,epsf,fancybox}
\usepackage{mathrsfs}
\usepackage{graphicx,graphics}
\usepackage{url}
\usepackage{epsf,epstopdf}
\usepackage{color}
\usepackage{amsmath,amsxtra,amsfonts,amscd,amssymb,bm}
\usepackage{multirow,cite}
\usepackage[algo2e,linesnumbered,vlined,ruled]{algorithm2e}

\numberwithin{table}{section}

\newtheorem{theorem}{Theorem}[section]

\newtheorem{lemma}[theorem]{Lemma}
\newtheorem{proposition}[theorem]{Proposition}
\newtheorem{definition}[theorem]{Definition}

\newtheorem{assumption}[theorem]{Assumption}
\newtheorem{remark}[theorem]{Remark}

\newcommand{\be}{\begin{equation}}
\newcommand{\ee}{\end{equation}}
\newcommand{\bee}{\begin{equation*}}
\newcommand{\eee}{\end{equation*}}
\newcommand{\bea}{\begin{eqnarray}}
\newcommand{\eea}{\end{eqnarray}}
\newcommand{\beaa}{\begin{eqnarray*}}
\newcommand{\eeaa}{\end{eqnarray*}}

\newcommand{\minimize}{\mathop{\textrm{min}}}
\newcommand{\st}{\textrm{s.t. }}
\newcommand{\xk}{x^k}
\newcommand{\xke}{x^{k+1}}
\newcommand{\xkm}{x^{k-1}}

\newcommand{\bD}{\mathbf{D}}
\newcommand{\bC}{\mathbf{C}}
\newcommand{\bT}{\mathbf{T}}
\newcommand{\ba}{\mathbf{a}}

\newcommand{\R}{\mathbb{R}}
\newcommand{\E}{{\bf\sf E}} 

\newcommand{\cM}{\mathcal{M}}
\newcommand{\cN}{\mathcal{N}}
\newcommand{\cT}{\mathcal{T}}

\newcommand{\cI}{\mathcal{I}}

\newcommand{\cL}{\mathcal{L}}

\newcommand{\dist}{\mathrm{dist}}

\newcommand{\grad}{\mathrm{grad}\, }
\newcommand{\Proj}{\mathrm{Proj}\, }
\newcommand{\RR}{\mathbb R}
\newcommand{\T}{\top}
\newcommand{\argmin}{\mathop{\rm argmin}}

\newcommand{\half}{\frac{1}{2}}
\newcommand{\etal}{{et al.\ }}

\begin{document}



\title{Primal-Dual Optimization Algorithms over Riemannian Manifolds: \\ an Iteration Complexity Analysis}

\author{
Junyu Zhang\thanks{Department of Industrial and System Engineering, University of Minnesota (zhan4393@umn.edu).}
\and
Shiqian Ma\thanks{Department of Mathematics, UC Davis (sqma@math.ucdavis.edu).}
\and
Shuzhong Zhang\thanks{Department of Industrial and System Engineering, University of Minnesota (zhangs@umn.edu).}
}

\date{October 5, 2017}

\maketitle

\begin{abstract}


In this paper we study nonconvex and nonsmooth multi-block optimization over Riemannian manifolds with coupled linear constraints. Such optimization problems naturally arise from machine learning, statistical learning, compressive sensing, image processing, and tensor PCA, among others. We develop an ADMM-like primal-dual approach based on decoupled solvable subroutines such as linearized proximal mappings. First, we introduce the optimality conditions for the afore-mentioned optimization models. Then, the notion of $\epsilon$-stationary solutions is introduced as a result. The main part of the paper is to show that the proposed algorithms enjoy an iteration complexity of $O(1/\epsilon^2)$ to reach an $\epsilon$-stationary solution. For prohibitively large-size tensor or machine learning models, we present a sampling-based stochastic algorithm with the same iteration complexity bound 
in expectation. In case the subproblems are not analytically solvable, a feasible curvilinear line-search variant of the algorithm based on retraction operators is proposed. Finally, we show specifically how the algorithms can be implemented to solve a variety of practical problems such as the NP-hard maximum bisection problem, the $\ell_q$ regularized sparse tensor principal component analysis and the community detection problem. Our preliminary numerical results show great potentials of the proposed methods.

\vspace{1cm}


\end{abstract}

\vspace{0.5cm}

\section{Introduction} \label{introduction}

Multi-block nonconvex optimization with nonsmooth regularization functions has recently found important applications in statistics, computer vision, machine learning, and image processing. In this paper, we aim to solve a class of {\it constrained}\/ nonconvex and nonsmooth optimization models. To get a sense of the problems at hand, let us consider the following {\it Multilinear (Tensor) Principal Component Analysis}\/ (MPCA) 
model, which has applications in 3-D object recognition, music genre classification, and subspace learning (see e.g.~\cite{MPCA1,MPCA2}). Details of the model will be discussed in Section~\ref{sec:application}. It pays to highlight here that a sparse optimization version of the model is as follows:
\[
\begin{array}{ll}
\min_{C,U,V,Y} & \sum_{i=1}^{N} \| T^{(i)}-C^{(i)}\times_1 U_1\times\cdots\times_dU_d\|_F^2 + \alpha_1\sum_{i=1}^{N}
\| C^{(i)}\|_p^p+\alpha_2\sum_{j=1}^{d}\|V_j\|_q^q+\frac{\mu}{2}\sum_{j=1}^{d}\|Y_j\|^2 \\
\st  & C^{(i)}\in\RR^{m_1\times\cdots\times m_d},\, i = 1,...,N  \\
     & U_j\in\RR^{n_j\times m_j},\,  U_j^\T U_j = I, j = 1,...,d \\
     & V_j - U_j+Y_j=0, \, j = 1,...,d,
\end{array}
\]
where $T^{(i)}\in \RR^{n_1\times\cdots\times n_d}$, $0<p<1$, $0<q<1$, $\alpha_1,\alpha_2,\mu>0$ are weighting parameters. Essentially, one aims to find a Tucker decomposition of a given tensor in such a way that the orthogonal matrices are sparse. This can be naturally dealt with by a consensus-variable approach; see for example \cite{Stiefel:Lai2014}. The factor matrices are introduced both as $U_j$ and $V_j$. While $U_j$'s are orthogonal (hence constrained to the Stiefel manifolds) and $V_j$'s are sparse, they are forced to agree with each other. This way of variable splitting is a useful modeling technique. Note that a slack variable $Y_j$ is introduced to relax this requirement. We penalize the norm of $Y_j$ in the objective so that $U_j$ and $V_j$ do not need to exactly equal to each other. Notice that the objective function involves sparsity-promoting nonconvex  $\ell_q$ $(0<q<1)$ loss functions. Therefore, the overall model is noncovex and nonsmooth because of the sparsity promoting objective function, in addition to the manifolds constraints. As we shall see from more examples later, such formulations are found to be common for many applications.

In general, we consider the following model: 
\bea
\label{prob:main}
& \min & f(x_1,\cdots,x_N) + \sum_{i = 1}^{N-1} r_i(x_i) \nonumber\\
& \st & \sum_{i = 1}^{N} A_ix_i = b, \mbox{ with } A_N = I, \nonumber \\
& & x_N\in\RR^{n_N}, \\
& & x_i \in \mathcal{M}_i, ~~ i = 1,...,N-1, \nonumber \\
& & x_i \in X_i, ~~ i = 1,...,N-1, \nonumber
\eea
where $f$ is a smooth function with $L$-Lipschitz continuous gradient, but is possibly nonconvex; the functions $r_i(x_i)$ are convex but are possibly nonsmooth;  $\mathcal{M}_i$'s are Riemannian manifolds, not necessarily compact, embedded in Euclidean spaces; the additional constraint sets $X_i$ are assumed to be some closed convex sets. As we shall see later, the restrictions on $r_i$ being convex and $A_N$ being identity can all be relaxed, after a reformulation. For the time being however, let us focus on \eqref{prob:main}.

\subsection{ Related literature }

On the modeling front, nonsmooth/nonconvex regularization such as the $\ell_1$ or $\ell_q$ ($0<q<1$) penalties are key ingredients in promoting sparsity in models such as the basis pursuit \cite{BasPurs,CompSens}, LASSO \cite{lasso, Bridge, ElaNet}, robust principal component analysis (RPCA) \cite{RPCA} and sparse coding \cite{SpCoding}.
Another important source for nonconvex modeling can be attributed to decomposition problems, e.g.\ tensor decomposition problems \cite{tensorD,Tucker,TTD},
low-rank and/or nonnegative matrix completion or decomposition \cite{lowR_MC,EM,nmf}.
Yet, another main source for nonconvex modeling is associated with
the Riemannian manifold constraints, such as sphere, product of spheres, the Stiefel manifold, the Grassmann manifold, and the low-rank elliptope are often encountered; see  
\cite{Opt_Manif:Absil-etal-2009,GeoStif,wenDouble,nemirv,dicR}.


There has been a recent intensive research interest in studying optimization over a Riemannian manifold:
\be
\min_{x\in\cM} f(x), \nonumber
\ee
where $f$ is smooth; see \cite{CG_NT,RM_GP,RM_NT,RM_tru,RM_tru1,RM_glo} and the references therein. Note that viewed {\it within}\/ manifold itself, the problem is essentially {\it unconstrained}.
Alongside deterministic algorithms, the stochastic gradient descent method (SGD) and the stochastic variance reduced gradient method (SVRG) have also been extended to optimization over Riemannian manifold; see e.g.  \cite{RM_sto1,RM_sto2,RM_sto3,RM_sto4,Jiang-SVRG-RM-2017}. 
Compared to all these approaches, our proposed methods allow a nonsmooth objective, a constraint $x_i\in X_i$, as well as the coupling affine constraints.
A key feature deviating from the traditional Riemannian optimization is that we take advantage of the global solutions for decoupled proximal mappings instead of relying on a retraction mapping, although if retraction mapping is available then it can be incorporated as well.

Alternating Direction Method of Multipliers (ADMM) has attracted much research attention in the past few years. Convergence and iteration complexity results have been thoroughly studied in the convex setting, and recently results have been extended to various nonconvex settings as well; see
\cite{NcvxADMM:Hong2016,NcvxADMM:HongLuoRa2016,NcvxADMM:Norate:LiPong2015,NcvxADMM:Norate:WangCaoXu2015,
NcvxADMM:Norate:WangYinZeng2015,NcvxADMM:Zhang-etal-2016,NcvxADMM:Norate:Nothm:YangPongChen2017}.
Among these results, \cite{NcvxADMM:Norate:LiPong2015,NcvxADMM:Norate:Nothm:YangPongChen2017,NcvxADMM:Norate:WangCaoXu2015,NcvxADMM:Norate:WangYinZeng2015} show the convergence to a stationary point without any iteration complexity guarantee. A closely related paper is \cite{NcvxADMM:ZhuZhang2017}, where the authors consider a multi-block nonconvex nonsmooth optimization problem on the Stiefel manifold with coupling linear constraints. An approximate augmented Lagrangian method is proposed to solve the problem and convergence to the KKT point is analyzed, but no iteration complexity result is given. Another related paper is \cite{NcvxADMM:Manifold-2015}, where the authors solve various manifold optimization problems with affine constraints by a two-block ADMM algorithm, without convergence assurance though. The current investigation is inspired by our previous work \cite{NcvxADMM:Zhang-etal-2016}, which requires the convexity of the constraint sets. In the current paper, we drop this restriction and extend the result to stochastic setting and allow Riemannian manifold constraints.
Speaking of nonconvex optimization, recent progress can be found under the name {\it nonsmooth and nonconvex composite optimization}; see~\cite{GhLan1,GhLan2,GhLan3,sto1}. However, in that case, the nonsmooth part of the objective and the constraint set are assumed to be convex, while these can be dropped in our approach as we noted earlier. 

Finally, we remark that for large-scale optimization such as tensor decomposition \cite{tensorD,Tucker,TTD}, black box tensor approximation problems \cite{TTcross,Tuckerbbox} and the worst-case input models estimation problems \cite{HenryLam1,HenryLam2}, the costs for function or gradient evaluation are prohibitively expensive. Our stochastic approach considerably alleviates the  computational burden. 


\subsection{Our contributions}

The contributions of this paper can be summarized as follows:
\begin{enumerate}
	\item [(i)] We define the $\epsilon$-stationary solution for problem \eqref{prob:main} with Riemmanian manifold constraints.
	\item [(ii)] We propose a nonconvex proximal gradient-based ADMM algorithm and its linearized variant, and analyze their iteration complexity to reach an $\epsilon$-stationary solution.
	\item [(iii)] We propose a stochastic variant of the nonconvex linearized proximal gradient-based ADMM with mini-batches, and establish its iteration complexity in the sense of expectation. 
	\item [(iv)] We propose a feasible curvilinear line-search variant of the nonconvex proximal gradient-based ADMM algorithm, where the exact minimization subroutine is replaced by a line-search procedure using a retraction operator. The iteration complexity of the method is established.

	\item [(v)] We present a number of extensions to the basic method, including relaxing the convexity of nonsmooth component of the objective, and relaxing
the condition on the last block matrix $A_N$. We also extend our analysis from Gauss-Seidel updating to Jacobi updating to enable parallel computing.
\end{enumerate}

\subsection{Organization of the paper}

The rest of the paper is organized as follows.
In Section \ref{sec:PrPr}, we review some basics of  Riemannian manifold. In the same section we derive the necessary optimality condition for a stationary point and the corresponding $\epsilon$-stationary solution for our optimization problem over Riemannian manifold. In Section \ref{sec:algo}, we propose a nonconvex proximal gradient-based ADMM and its three
variants with iteration complexity bounds. In Section \ref{sec:Extension}, we present extensions of our basic model. In Section \ref{sec:application}, we present the implementations of our approach to nonnegative sparse tensor decomposition, the maximum bisection problem, and sparse MPCA. Finally, in Section \ref{sec:Num_rst} we present the results of numerical experiments.
For the ease of presentation, the proofs of technical lemmas are delegated to the appendix.

\section{ Optimality over Manifolds }\label{sec:PrPr}

In this section,
we shall introduce the basics of optimization over manifolds.
The discussion is intended as background information for our purpose; thorough treatments on the topic can be found in, e.g.\
\cite{Smooth_Manif:Lee-John-2008,Opt_Manif:Absil-etal-2009}.
We then extend the first-order optimality condition for constrained optimization on manifold established in \cite{RieOpt:Yang-etal-2012} to our constrained model \eqref{prob:main}. Based on the optimality condition, we introduce the notion of $\epsilon$-stationary solution, and $\epsilon$-stationary solution in expectation (for the stochastic setting) respectively.

Suppose $\cM$ is a differentiable manifold, then for any $x\in\cM$, there exists a \emph{chart} $(U,\psi)$ in which $U$ is an open set with $x\in U\subset\cM$ and $\psi$ is a homeomorphism between $U$ and an open set $\psi(U)$ in Euclidean space. This coordinate transform enables us to locally treat a Riemannian manifold as a Euclidean space. Denote the tangent space $\cM$ at point $x\in\cM$ by $\cT_x\cM$, then $\cM$ is a Riemannian manifold if it is equipped with a metric on the tangent space $\cT_x\cM$ which is continuous in $x$.
\begin{definition}[Tangent Space]\label{Tspace}
Consider a Riemannian manifold $\cM$ embedded in a Euclidean space. For any $x\in\cM$, the tangent space $\cT_x\cM$ at $x$ is a linear subspace consists of the derivatives 
of all smooth curves on $\cM$ passing $x$; that is
	\be
	\cT_x\cM = \left\{\gamma'(0): \gamma(0) = x, \gamma([-\delta,\delta])\subset\cM, \mbox{ for some } \delta>0, \gamma\mbox{ is smooth}\right\}.
	\ee
	The Riemannian metric, i.e., the inner product between $u,v\in\cT_x\cM$, is defined to be $\langle u,v \rangle_x := \langle u,v \rangle$, where the latter is the Euclidean inner product.
\end{definition}

Define the set of all functions differentiable at point $x$ to be $\mathcal{F}_x$. An alternative but more general way of defining tangent space is by viewing a tangent vector $v\in\cT_x\cM$ as an operator mapping $f\in\mathcal{F}_x$ to $v[f]\in\RR$ which satisfies the following property: For any given $f\in\mathcal{F}_x$, there exists a smooth curve $\gamma$ on $\cM$  with $\gamma(0) = x$ and
$v[f] = \frac{d(f(\gamma(t)))}{dt} \bigg{|}_{t = 0}$.
For manifolds embedded in Euclidean spaces, we can obtain  Definition \ref{Tspace} by defining $v = \gamma'(0)$ and $v[f] = \langle \gamma'(0),\nabla f(x)\rangle$.

For example, when $\cM$ is a sphere, $\cT_x\cM$ is the tangent plane at $x$ with a proper translation such that the origin is included. When $\cM = \RR^n$, then $\cT_x\cM = \RR^n = \cM$.

\begin{definition} [Riemannian Gradient]
	For $f\in\mathcal{F}_x$, the Riemannian gradient $\grad f(x)$ is a tangent vector in $\cT_x\cM$ satisfying
	$v[f] = \langle v,\grad f(x)\rangle_x \mbox{ for any } v\in\cT_x\cM.$
	
\end{definition}
If $\cM$ is an embedded submanifold of a Euclidean space, we have
$$\grad f(x) = \Proj_{\cT_x\cM}(\nabla f(x)),$$
where $\Proj_{\cT_x\cM}$ is the Euclidean projection operator onto the subspace $\cT_x\cM$, which is a nonexpansive linear transformation.

\begin{definition} [Differential]
	Let $F: \cM\rightarrow\cN$ be a smooth mapping between two Riemannian manifolds $\cM$ and $\cN$. The differential (or push-forward) of $F$ at $x$ is a mapping $\bD F(x):\cT_x\cM\rightarrow\cT_{F(x)}\cN$  defined by
	$$(\bD F(x)[v])[f] = v[f\circ F], \mbox{ for all $v\in\cT_x\cM$, and $\forall f\in\mathcal{F}_{F(x)}$}.$$
\end{definition}

Suppose $\cM$ is an $m$-dimensional embedded Riemannian submanifold of $\RR^n, m\leq n$, and let $(U,\psi)$ be a chart at point $x\in\cM$, then $\psi$ is a smooth mapping from $U\subset\cM$ to $\psi(U)\subset\cN = \RR^m$. Under a proper set of basis  $\{\ba_i\}_{i=1}^m$ of $\cT_x\cM$ and suppose $v = \sum_{i=1}^mv_i\ba_i$, then
$$\hat{v} :=\bD \psi(x) [v] =  (v_1,...,v_m).$$
Clearly, this establishes a bijection between the tangent space $\cT_x\cM$ and the tangent space of $\cT_{\psi(x)}\psi(U) = \RR^m$. Following the notation in \cite{RieOpt:Yang-etal-2012}, we use $\hat{o}$ to denote the Euclidean counterpart of an object $o$ in $\cM$; e.g.,
$$\hat{f} = f\circ\psi^{-1},~~~~ \hat{v} = \bD\psi(x)[v],~~~~ \hat{x} = \psi(x).$$
Finally, if we define the Gram matrix $G_{x}(i,j) = \langle \ba_i,\ba_j\rangle_x$, which is also known as the Riemannian metric, then $\langle u,v\rangle_x = \hat{u}^\T G_{x}\hat{v}.$

Next, we shall present a few optimization concepts generalized to the manifold case. Let $C$ be a subset in $\RR^n$ and $x\in C$, the tangent cone $T_C(x)$ and the normal cone $N_C(x)$ of $C$ at $x$ are defined in accordance with that in \cite{Nocedal1999}. Suppose $S$ is a closed subset on the Riemannian manifold $\cM$, $(U,\psi)$ is a chart at point $x\in S$, then by using coordinate transform (see also \cite{RieOpt:Yang-etal-2012, Motreanu1982Quasi}), the Riemannian tangent cone can be defined as
\be
\label{Tcone}
\cT_S(x) := [\bD\psi(x)]^{-1}[T_{\psi(S\cap U)}(\psi(x))].
\ee
Consequently, the Riemannian normal cone can be defined as
\be
\label{Ncone}
\cN_S(x) := \{u\in\cT_x\cM:\langle u,v\rangle_x\leq 0, \forall v\in\cT_S(x)\}.
\ee
By a rather standard argument (see \cite{RieOpt:Yang-etal-2012}), the following proposition can be shown:
\begin{proposition}
	\label{NCone_E2R}
	$\cN_S(x) = [\bD\psi(x)]^{-1}[G_{x}^{-1}N_{\psi(U\cap S)}(\psi(x))].$
\end{proposition}

A function $f$ is said to be locally Lipschitz on $\cM$ if for any $x\in\cM$, there exists some $L>0$ such that in a neighborhood of $x$, $f$ is $L$-Lipschitz in the sense of Riemannian distance. When $\cM$ is a compact manifold, a global $L$ exists. When $\cM$ is an embedded submanifold of $\RR^n$ and $f$ is a locally Lipschitz on $\RR^n,$ let $f|_{\cM}$ be  the function $f$ restricted to $\cM$, then $f|_{\cM}$ is also locally Lipschitz on $\cM$.

\begin{definition}[The Clarke subdifferential on Riemannian manifold \cite{RieOpt:Yang-etal-2012,HP2011}]
	For a locally Lipschitz continuous function $f$ on $\cM$, the \emph{Riemannian generalized directional derivative} of $f$ at $x\in\cM$ in direction $v\in\cT_x\cM$ is defined as
	\be
	\label{Direc_Dir}
	f^{\circ}(x;v) = \limsup_{y\rightarrow x,t\downarrow0}\frac{f\circ\psi^{-1}(\psi(y)+t\bD\psi(y)[v])-f\circ\psi^{-1}}{t}.
	\ee
	Then the Clarke subdifferential is defined as
	\be
	\label{Clarke_Sub}
	\partial f(x) = \{\xi\in\cT_x\cM:\langle\xi,v\rangle\leq f^{\circ}(x;v), \forall v\in\cT_x\cM\}.
	\ee
\end{definition}

There are several remarks for the notion of Riemannian Clarke subdifferentials. If $\cM = \RR^n$ and $\psi = id$, then the above notion reduces to the original Clarke subdifferential \cite{Clarke1983}. In this case, suppose $f$ is differentiable and $r$ is locally Lipschitz, then we have
\be
\label{Clarke:f+r}
\partial(f+r)(x) = \nabla f(x)+\partial r(x),
\ee
where $\partial r(x)$ is the Clarke subdifferential. Furthermore, if we have additional manifold constraints and $r$ is convex, from \cite{RieOpt:Yang-etal-2012} we have
\be
\label{proj_sub}
\partial (f+r)|_{\cM}(x) = \Proj_{\cT_x\cM}(\nabla f(x)+\partial r(x)).
\ee
The convexity of $r$ is crucial in this property. If the nonsmooth part $r_i(x_i)$ in our problem is also nonconvex, then we will have to use additional variables and consensus constraints to decouple $r_i$, the manifold constraint and smooth component $f$, which will be discussed in Section \ref{sec:Extension}. More importantly, we have the following result (see \cite{RieOpt:Yang-etal-2012}):
\begin{proposition}
	\label{SubDif_E2R}
	Suppose $f$ is locally Lipschitz continuous in a neighborhood of $x$, and $(U,\psi)$ is a chart at $x$. It holds that
	$$\partial f(x) = [\bD\psi(x)]^{-1}[G_{x}^{-1}\partial (f\circ\psi^{-1})(\psi(x))].$$
\end{proposition}

\subsection{Optimality condition and the  $\epsilon$-stationary solution}
Consider the following optimization problem over manifold:
\bea
\label{prob:Rm1}
& \minimize & f(x)  \\
& \st & x\in S\subset \cM.   \nonumber
\eea
Suppose that $x^*$ is a local minimum, and that $(U,\psi)$ is a chart at $x^*$. Then, $\hat{x}^*:=\psi(x^*)$ must also be a local minimum for the problem
\bea
\label{prob:Ec1}
& \minimize & \hat{f}(\hat{x})  \\
& \st & \hat{x}\in \psi(S\cap U).   \nonumber
\eea
Therefore, problem \eqref{prob:Rm1} is transformed into a standard nonlinear programming problem \eqref{prob:Ec1} in Euclidean space. We will then find the optimality condition via \eqref{prob:Ec1} and map it back to that of \eqref{prob:Rm1} by using the differential operator.

Assume that both $\hat{f}$ and $f$ are locally Lipschitz. The optimality of $\hat{x}^*$ yields (cf.~\cite{Clarke1983})
$$
0\in\partial \hat{f}(\hat{x}^*)+N_{\psi(U\cap S)}(\hat{x}^*).
$$
Apply the bijection $[\bD\psi(x)]^{-1}\circ G_{x}^{-1}$ on both sides, and by Propositions \ref{SubDif_E2R} and \ref{NCone_E2R}, the first-order optimality condition for problem \eqref{prob:Rm1} follows as a result:
\be
\label{opt_cond_1}
0\in\partial f(x^*)+\cN_S(x^*).
\ee
If $f$ is differentiable, then \eqref{opt_cond_1} reduces to
$$-\grad f(x^*)\in\cN_S(x^*).$$
To specify the set $S$ in problem \eqref{prob:main},
let us consider an equality constrained problem
\bea
\label{prob:eq_cons}
& \minimize & f(x)\nonumber\\
& \st & c_i(x) = 0,i = 1,...,m, \\
& & x\in \cM \cap X. \nonumber
\eea
Note that in the case of \eqref{prob:main}, the above constraints $c_i(x)=0$, $i=1,2,...,m$, represent the linear equality constraints.
Define $\Omega := \{x\in\cM: c_i(x) = 0,\, i=1,...,m\}$, and $S := \Omega\cap X$. By assuming the so-called Linear Independent Constraint Qualification (LICQ) condition on the Jacobian of $c(x)$ at $x^*$, Corollary 4.1 in \cite{RieOpt:Yang-etal-2012} implies
\be
\label{Ncone_1}
\cN_{\Omega}(x^*) = \left. \left\{\sum_{i=1}^m \lambda_i\, \grad c_i(x^*) \, \right| \, \lambda\in \RR^m \right\} = -(\cT_{\Omega}(x^*))^\star,
\ee
where ${\cal K}^\star$ indicates the dual of cone ${\cal K}$. 
Therefore, \eqref{opt_cond_1} implies
$$
\partial f(x^*) \cap \left( -\cN_S(x^*)  \right) \not= \emptyset .
$$
We have
\bea
-(\cN_{\Omega}(x^*)+\cN_{X}(x^*)) & = & (\cT_{\Omega}(x^*))^\star+(\cT_{X}(x^*))^\star \nonumber\\
& \subseteqq &
\mathrm{cl}\, ((\cT_{\Omega}(x^*))^\star+(\cT_{X}(x^*))^\star) \nonumber\\
& = & (\cT_{\Omega}(x^*)\cap\cT_{X}(x^*))^\star\nonumber\\
& \subseteqq&
(\cT_{\Omega\cap X}(x^*))^\star. \nonumber
\eea
  The optimality condition is established as:
\begin{proposition}
	\label{opt_cond_eq}
Suppose that $x^*\in \cM \cap X$ and $c_i(x^*)= 0, i = 1,...,m$. If
	$$
    \partial f(x^*) \cap \left( - \cN_{\Omega}(x^*) - \cN_X (x^*)  \right)  \not=\emptyset,
    $$
then $x^*$ is a stationary solution for problem \eqref{prob:eq_cons}.
\end{proposition}


By specifying the optimality condition in Proposition \ref{opt_cond_eq} to \eqref{prob:main}, we have:
\begin{theorem} Consider problem \eqref{prob:main} where $f$ is smooth with Lipschitz gradient and $r_i$'s are convex and locally Lipschitz continuous. If there exists a Lagrange multiplier $\lambda^*$ such that
	\be
	\label{opt_ADMM}
	\begin{cases}  \nabla_Nf(x^*)-A_N^\T \lambda^* =0,\\
		\sum_{i = 1}^{N}A_ix_i^*-b=0,\\
		\Proj_{\cT_{x_i^*}\mathcal{M}_i}\left(\nabla_i f(x^*)-A_i^\T \lambda^*+\partial r_i(x_i^*)\right)+\mathcal{N}_{X_i\cap \cM_i}(x_i^*)\ni0, i = 1,...,N-1,
	\end{cases}
	\ee
then $x^*$ is a stationary solution for problem \eqref{prob:main}.
\end{theorem}
Hence, an $\epsilon$-stationary solution of problem \eqref{prob:main} can be naturally defined as:
\begin{definition} [$\epsilon$-stationary solution]
	\label{def:epsolu} Consider problem \eqref{prob:main} where $f$ is smooth with Lipschitz gradient and $r_i$ are convex and locally Lipschitz continuous. Solution $x^*$ is said to be an $\epsilon$-stationary solution if there exists a multiplier $\lambda^*$ such that
	\be
	\begin{cases} \|\nabla_Nf(x^*)-A_N^\T \lambda^*\|\leq\epsilon,\\
		\|\sum_{i = 1}^{N}A_ix_i^*-b\|\leq\epsilon,\nonumber\\
		\dist\left(\Proj_{\cT_{x_i^*}\mathcal{M}_i}\left(-\nabla_i f(x^*)+A_i^\T \lambda^*-\partial r_i(x_i^*)\right),\mathcal{N}_{X_i\cap\mathcal{M}_i}(x_i^*)\right) 
\leq\epsilon, i = 1,...,N-1.
	\end{cases}
	\ee
\end{definition}

In the case that $x^*$ is a vector generated by some randomized algorithm, the following adaptation is appropriate.

\begin{definition} [$\epsilon$-stationary solution in expectation]\label{def:epsolu-exp} Suppose that $x^*$ and $\lambda^*$ are generated by some randomized process. Then, we call $x^*$ and $\lambda^*$ to be $\epsilon$-stationary solution for problem \eqref{prob:main} in expectation if the following holds
	\be
	\begin{cases} \E\left[\|\nabla_Nf(x^*)-A_N^\T \lambda^*\|\right]\leq\epsilon,\\ \smallskip
		\E\left[\|\sum_{i = 1}^{N}A_ix_i^*-b\|\right]\leq\epsilon,\nonumber \\ \smallskip
		\E \left[\dist \left(\Proj_{\cT_{x_i^*}\mathcal{M}_i}\left( -\nabla_if(x^*)+A_i^\T \lambda^*-\partial r_i(x_i^*)\right),\mathcal{N}_{X_i\cap\mathcal{M}_i}(x_i^*)\right)\right]\leq\epsilon, i = 1,...,N-1.
	\end{cases}
	\ee
\end{definition}

\section{Proximal Gradient ADMM and Its Variants}
\label{sec:algo}

In \cite{NcvxADMM:Zhang-etal-2016}, Jiang, Lin, Ma and Zhang proposed a proximal gradient-based variant of ADMM for nonconvex and nonsmooth optimization model with convex constraints. 
In this paper, we extend the analysis to include nonconvex Riemannian manifold constraints, motivated by the vast array of potential applications.
Moreover, we propose to linearize the nonconvex function $f$, which significantly broadens the applicability and enables us to utilize the stochastic gradient-based method to reduce computational costs for large-scale problems. As it turns out, the convergence result for this variant remains intact.

Concerning problem \eqref{prob:main}, we first make some assumptions on $f$ and $r_i$'s.
\begin{assumption}
	\label{assumption-1-Lbounds}
	$f$ and $r_i, i = 1,...,N-1$, are all bounded from bellow in the feasible region. We denote the lower bounds by $r_i^* = \min_{x_i\in\cM_i\cap X_i}r_i(x_i), i = 1,...,N-1$ and
	$$f^* = \min_{x_i\in\cM_i\cap X_i, i=1,...,N-1, x_N\in\RR^{n_N}}f(x_1,\cdots,x_N).$$
\end{assumption}

\begin{assumption}
	\label{assumption-2-Lips}
	$f$ is a smooth function with $L$-Lipschitz continuous gradient; i.e.
    \be\label{eq:assumption-2-lips}
    \|\nabla f(x_1,\ldots,x_N)-\nabla f(\hat{x}_1,\ldots,\hat{x}_N)\|_2 \leq L\|(x_1-\hat{x}_1,\ldots,x_N-\hat{x}_N)\|_2, \,\,\, \forall x, \hat{x}.
    \ee
\end{assumption}

\begin{assumption}
	\label{assumption-3-subP_globalsolu}
The proximal mappings required at Step 1 of Algorithms \ref{alg:PADMM}, \ref{alg:PADMM-L} and \ref{alg:PADMM-S} are all computable. (As we will see in Section \ref{sec:application}, this assumption holds true for many practical applications).
\end{assumption}

\subsection{Nonconvex proximal gradient-based ADMM}

The augmented Lagrangian function for problem \eqref{prob:main} is
\be
\label{Lagrangian}
\cL_{\beta}(x_1,x_2,\cdots,x_N,\lambda) = f(x_1,\cdots,x_N)+\sum_{i=1}^{N-1}r_i(x_i)-\bigg\langle \sum_{i=1}^{N}A_ix_i-b,\lambda \bigg\rangle + \frac{\beta}{2}\left\|\sum_{i=1}^{N}A_ix_i-b\right\|^2,
\ee
where $\lambda$ is the Lagrange multiplier, $\beta>0$ is a penalty parameter. Our proximal gradient-based ADMM for solving \eqref{prob:main} is described in Algorithm \ref{alg:PADMM}.

\begin{algorithm2e}[H]
	\caption{Nonconvex Proximal Gradient-Based ADMM on Riemannian Manifold}
	\label{alg:PADMM}
	Given $(x_1^0,x_2^0,\cdots,x_N^0)\in(\mathcal{M}_1\cap X_1)\times(\mathcal{M}_2\cap X_2)\times\cdots \times(\mathcal{M}_{N-1}\cap X_{N-1})\times\RR^{n_N}$, $\lambda^0\in \RR^m$, $\beta>0$, $\gamma>0$, $H_i\succ 0, i=1,\ldots,N-1$.\\
	\For{$k = 0,1,...$ }{
		$[\mbox{Step 1}]$ For $i = 1,2,...,N-1$, and positive semi-definite matrix $H_i$, compute
$x_i^{k+1} := \argmin_{x_i\in \mathcal{M}_i\cap X_i }\mathcal{L}_{\beta}(x_1^{k+1},\cdots,x_{i-1}^{k+1},x_i,x_{i+1}^{k},\cdots,x_N^k,\lambda^k)+\frac{1}{2}\|x_i-x_i^k\|^2_{H_i}$; \\
		$[\mbox{Step 2}]$ $x_{N}^{k+1} := x_N^k-\gamma\nabla_N\mathcal{L}_{\beta}(x_1^{k+1},\cdots,x_{N-1}^{k+1},x_N^k,\lambda^k)$; \\
		$[\mbox{Step 3}]$ $\lambda^{k+1} := \lambda^k-\beta(\sum_{i = 1}^{N}A_ix_i^{k+1}-b)$.
	}
\end{algorithm2e}

Before we give the main convergence result of Algorithm \ref{alg:PADMM}, we need the following lemmas. Lemmas \ref{lm:PADMM-lemma1} and \ref{lm:PADMM-lemma3} are from \cite{NcvxADMM:Zhang-etal-2016}; 
and the proof of Lemma \ref{lm:PADMM-lemma2} is in the appendix.

\begin{lemma} (Lemma 3.9 in \cite{NcvxADMM:Zhang-etal-2016})
	\label{lm:PADMM-lemma1} Suppose that the sequence $\{x_1^k,...,\xk_N,\lambda^k\}$ is generated by Algorithm \ref{alg:PADMM}. Then,
	\bea
	\|\lambda^{k+1}-\lambda^k\|^2 & \leq & 3\left(\beta-\frac{1}{\gamma}\right)^2\|\xk_N-\xke_N\|^2+3\left[\left(\beta-\frac{1}{\gamma}\right)^2+L^2\right]\|\xkm_N-\xk_N\|^2 \nonumber \\
	& & +3L^2\sum_{i=1}^{N-1}\|\xk_i-\xke_i\|^2.\label{lemma-bound}
	\eea
\end{lemma}
Since Steps 2 and 3 in Algorithm \ref{alg:PADMM} are the same as those in \cite{NcvxADMM:Zhang-etal-2016}, this lemma remains valid here. Specially, Step 2 and Step 3 directly result in
\be
\label{To-hard-to-give-a-name-TAT}
\lambda^{k+1} = \left(\beta-\frac{1}{\gamma}\right)(\xk_N-\xke_N)+\nabla_Nf(\xke_1,\ldots,\xke_{N-1},\xk_N).
\ee
We define a potential function
\be
\label{Decrease_func}
\Psi_G(x_1,\cdots,x_N,\lambda,\bar{x}) = \cL_\beta(x_1,\cdots,x_N,\lambda) + \frac{3}{\beta}\left[\left(\beta-\frac{1}{\gamma}\right)^2+L^2\right]\|\bar{x}-x_N\|^2.
\ee

With Lemma \ref{lm:PADMM-lemma1}, the following monotonicity property can be established.

\begin{lemma}\label{lm:PADMM-lemma2}
	Suppose the sequence $\{(x^k_1,\cdots, x^k_N,\lambda_k)\}$ is generated by Algorithm \ref{alg:PADMM}. Assume that
	\be
	\label{beta}
	\beta > 
	\left(\frac{6+18\sqrt{3}}{13}\right)L \approx 2.860L \mbox{ and } H_i\succ\frac{6L^2}{\beta}I, i = 1,...,N-1.
	\ee
	Then $\Psi_G(\xke_1,\cdots,\xke_N,\lambda^{k+1},\xk_N)$ is monotonically decreasing over $k$ if $\gamma$ lies in the following interval:
	\be
	\label{gamma}
	\gamma\in
	\left(\frac{12}{13\beta+\sqrt{13\beta^2-12\beta L-72L^2}},\frac{12}{13\beta-\sqrt{13\beta^2-12\beta L-72L^2}}\right).
	\ee
\end{lemma}
More specifically, we have
\bea
\label{lm_:PADMM-lemma2:5}
& & \Psi_G(\xke_1,\cdots,\xke_{N-1},\xke_N,\lambda^{k+1},\xk_N) - \Psi_G(\xk_1,\cdots,\xk_{N-1},\xk_N,\lambda^{k},\xkm_N) \nonumber\\
&\leq& \left[\frac{\beta+L}{2}-\frac{1}{\gamma}+\frac{6}{\beta}\left(\beta-\frac{1}{\gamma}\right)^2+\frac{3L^2}{\beta}\right]\|\xk_N-\xke_N\|^2 \\
& & - \sum_{i = 1}^{N-1}\|\xk_i-\xke_i\|^2_{\frac{1}{2}H_i-\frac{3L^2}{\beta}I} < 0.\nonumber
\eea


\begin{lemma}
	\label{lm:PADMM-lemma3}
	(Lemma 3.11 in \cite{NcvxADMM:Zhang-etal-2016}) Suppose that the sequence $\{\xk_1,\cdots,\xk_N,\lambda^k\}$ is generated by Algorithm \ref{alg:PADMM}. It holds that
	\be
	\Psi_G(\xke_1,\cdots,\xke_{N-1},\xke_N,\lambda^{k+1},\xk_N) \geq \sum_{i=1}^{N-1}r_i^*+f^*,
	\ee
	where $r_i^*, i = 1,...,N-1$ and $f^*$ are defined in Assumption \ref{assumption-1-Lbounds}.
\end{lemma}

Denote $\sigma_{\min}(M)$ as the smallest singular value of a matrix $M$. Now we are ready to present the main convergence result of Algorithm \ref{alg:PADMM}.
\begin{theorem}\label{thm:PADMM}
	Suppose that the sequence $\{\xk_1,...,\xk_N,\lambda^k\}$ is generated by Algorithm \ref{alg:PADMM}, and the parameters $\beta$ and $\gamma$ satisfy \eqref{beta} and \eqref{gamma} respectively. Define $\kappa_1 := \frac{3}{\beta^2}\left[\left(\beta-\frac{1}{\gamma}\right)^2+L^2\right]$, $\kappa_2 := \left(|\beta-\frac{1}{\gamma}|+L\right)^2$, $\kappa_3 := \left(L+\beta\sqrt{N}\max_{1\leq i\leq  N}  \|A_i\|_2^2 +\max_{1\leq i\leq N-1}\|H_i\|_2\right)^2$ and \\
	$\tau:= \min\left\{ -\left[\frac{\beta+L}{2}-\frac{1}{\gamma}+\frac{6}{\beta}\left(\beta-\frac{1}{\gamma}\right)^2+\frac{3L^2}{\beta}\right], \min_{i = 1,...,N-1}\left[-\left(\frac{3L^2}{\beta}-\frac{\sigma_{\min}(H_i)}{2}\right) \right] \right\} $. Assuming  $H_i\succ\frac{6L^2}{\beta}I$ and letting
\be\label{def:K-retraction}
K := \left\lceil\frac{2\max\{\kappa_1,\kappa_2,\kappa_3\}}{\tau\epsilon^2}
\left(\Psi_G(x_1^1,...,x_N^1,\lambda^1,x_N^0)-\sum_{i=1}^{N-1}r_i^*-f^*\right)\right\rceil,
\ee
and $k^* := \argmin_{2\leq k\leq K+1}\sum_{i=1}^N(\|\xk_i-\xke_i\|^2+\|\xkm_i-\xk_i\|^2),$ it follows that $(x_1^{k^*+1},\cdots,x_N^{k^*+1},\lambda^{k^*+1})$ is an $\epsilon$-stationary solution of \eqref{prob:main} defined in Definition \ref{def:epsolu}.
\end{theorem}

\begin{proof}
For the ease of presentation, we denote
\be
\label{thm2_1}
\theta_k := \sum_{i=1}^N(\|\xk_i-\xke_i\|^2+\|\xkm_i-\xk_i\|^2).
\ee
Summing \eqref{lm_:PADMM-lemma2:5} over $k=1,\ldots,K$ yields
\be
\label{thm2_2}
\Psi_G(x_1^1,\cdots,x_N^1,\lambda^1,x_N^0) - \Psi_G(x_1^{K+1},\cdots,x_N^{K+1},\lambda^{K+1},x_N^K)\geq \tau\sum_{k=1}^{K}\sum_{i=1}^N\|\xk_i-\xke_i\|^2,
\ee
which implies
\bea
& & \min_{2\leq k\leq K+1} \theta_k \nonumber\\
& \leq &
\frac{1}{\tau K}\left[2\Psi_G(x_1^1,\cdots,x_N^1,\lambda^1,x_N^0) - \Psi_G(x_1^{K+1},\ldots,x_N^{K+1},\lambda^{K+1},x_N^K) - \Psi_G(x_1^{K+2},\ldots,x_N^{K+2},\lambda^{K+2},x_N^{K+1})\right] \nonumber \\
& \leq& \frac{2}{\tau K}\left[\Psi_G(x_1^1,\cdots,x_N^1,\lambda^1,x_N^0) - f^* -\sum_{i = 1}^{N-1}r_i^*\right]. \label{thm2_3}
\eea
By \eqref{To-hard-to-give-a-name-TAT} we have
\bea
& &\|\lambda^{k+1} - \nabla_N f(\xke_1,\cdots,\xke_N)\|^2 \nonumber\\
&\leq& \left( \left|\beta-\frac{1}{\gamma}\right| \|\xk_N-\xke_N\|+\|\nabla_Nf(\xke_1,\cdots,\xke_{N-1},\xk_N)-\nabla_Nf(\xke_1,\cdots,\xke_N)\|\right)^2\nonumber\\
&\leq&  \left(\left| \beta-\frac{1}{\gamma} \right|+L\right)^2\|\xk_N-\xke_N\|^2 \nonumber\\
&\leq & \kappa_2 \theta_k. \label{ConsVio}
\eea
From Step 3 of Algorithm \ref{alg:PADMM} and \eqref{lemma-bound}, we have
\bea
& & \left\|\sum_{i=1}^{N-1}A_i\xke_i+\xke_N-b\right\|^2   \nonumber \\
&=&  \frac{1}{\beta^2}\|\lambda^k-\lambda^{k+1}\|^2 \nonumber  \\
& \leq & \frac{3}{\beta^2}\left[\left(\beta-\frac{1}{\gamma}\right)^2+L^2\right]\|x^{k-1}_N-\xk_N\|^2 + \frac{3}{\beta^2}\left(\beta-\frac{1}{\gamma}\right)^2\|\xke_N-\xk_N\|^2
+ \frac{3L^2}{\beta^2}\sum_{i=1}^{N-1}\|\xke_i-\xk_i\|^2 \nonumber \\
&\leq& \kappa_1 \theta_k. \label{lambda}
\eea

By the optimality conditions (e.g., \eqref{opt_cond_1}) for the subproblems in Step 1 of Algorithm \ref{alg:PADMM}, and using \eqref{proj_sub} and Step 3 of Algorithm \ref{alg:PADMM}, we can get
\bea
\Proj_{\mathcal{T}_{x_i^{k+1}}\mathcal{M}_i}\biggl\{\nabla_if(x_1^{k+1},\cdots,x_i^{k+1},x_{i+1}^k,\cdots,x_{N}^k)-A_i^\T \lambda^{k+1}+\beta A_i^\T \left(\sum_{j = i+1}^{N}A_j(x_j^k-x_j^{k+1})\right) \nonumber\\
+H_i(x_i^{k+1}-x_i^k)+g_i(x_i^{k+1})\biggr\}+q_i(x_i^{k+1}) = 0, \label{subPopt}
\eea

for some $g_i(x_i^{k+1})\in\partial r_i(x_i^{k+1})$, $q_i(x_i^{k+1})\in\mathcal{N}_{X_i}(x_i^{k+1})$. Therefore,

\begin{eqnarray*}
& &  \dist\left(\Proj_{\cT_{x_i^{k+1}}\mathcal{M}_i}\biggl\{-\nabla_i f(x^{k+1})+A_i^\T \lambda^{k+1}-\partial r_i(x_i^{k+1})\biggr\},\mathcal{N}_{X_i}(x_i^{k+1})\right) \nonumber \\
&\leq& \biggl\|\Proj_{\cT_{x_i^{k+1}}\mathcal{M}_i}\biggl\{-\nabla_i f(x^{k+1})+A_i^T\lambda^{k+1}-g_i(x_i^{k+1})-q_i(x_i^{k+1})\biggr\}\biggr\| \nonumber\\
& = & \biggl\|\Proj_{\cT_{x_i^{k+1}}\mathcal{M}_i}\biggl\{-\nabla_i f(x^{k+1})+\nabla_if(x_1^{k+1},\cdots, x_i^{k+1},x_{i+1}^k,\cdots,x_{N}^k) \nonumber\\
& &+\beta A_i^\T (\sum_{j = i+1}^{N}A_j(x_j^k-x_j^{k+1}))
+H_i(x_i^{k+1}-x_i^k)\biggr\}\biggr\|  \nonumber\\
\end{eqnarray*}

\begin{eqnarray}
&\leq& \|-\nabla_if(x^{k+1})+\nabla_if(x_1^{k+1},\cdots,x_i^{k+1},x_{i+1}^k,\cdots,x_{N}^k)-H_i(x_i^{k+1}-x_i^k) \nonumber\\
& &+\beta A_i^\T (\sum_{j = i+1}^{N}A_j(x_j^{k+1}-x_j^k)) \| \nonumber\\
&\leq& \|\nabla_if(\xke)-\nabla_if(\xke_1,\cdots,\xke_i,\xk_{i+1},\cdots,\xk_N)\| + \|H_i(x_i^{k+1}-x_i^k)\| \nonumber \\
& & +\|\beta A_i^\T (\sum_{j = i+1}^{N}A_j(x_j^{k+1}-x_j^k)) \| \nonumber \\
& \leq & \left(L+\beta\max_{1\leq j\leq N}\|A_j\|_2^2\sqrt{N}\right)\sqrt{\sum_{j = i+1}^N\|\xke_j-\xk_j\|^2} + \max_{1\leq j\leq N-1}\|H_j\|_2\|\xk_i-\xke_i\|  \nonumber\\
& \leq & \sqrt{\kappa_3\theta_k}. \label{otherblock}
\end{eqnarray}

Combining \eqref{ConsVio}, \eqref{lambda}, \eqref{otherblock} and \eqref{def:K} yields the desired result. \end{proof}

\subsection{Nonconvex linearized proximal gradient-based ADMM}
When modeling nonconvex and nonsmooth optimization with manifold constraints, it is often the case that computing proximal mapping (in the presence of $f$) may be difficult, while optimizing with a quadratic objective is still possible. This leads to a variant of ADMM which linearizes the $f$ function.
In particular, we define the following approximation to the augmented Lagrangian function: 
\bea
\label{LagApprox}
\hat{\cL}_{\beta}^i(x_i;\hat{x}_1,\cdots,\hat{x}_N,\lambda) & := & f(\hat{x}_1,\cdots,\hat{x}_N)+\langle \nabla_i f(\hat{x}_1,\cdots,\hat{x}_N),x_i-\hat{x}_i\rangle + r_i(x_i) \nonumber\\
& & -\bigg\langle \sum_{j=1,j\neq i}^{N}A_j\hat{x}_j+ A_ix_i-b,\lambda\bigg\rangle +\frac{\beta}{2}\bigg\|\sum_{j=1,j\neq i}^{N}A_j\hat{x}_j+ A_ix_i-b\bigg\|^2,
\eea
where $\lambda$ is the Lagrange multiplier and $\beta>0$ is a penalty parameter. It is worth noting that this approximation is defined with respect to a particular block of variable $x_i$. The linearized proximal gradient-based ADMM algorithm is described as in Algorithm \ref{alg:PADMM-L}.

\begin{algorithm2e}[H]
	\caption{Nonconvex Linearized Proximal Gradient-Based ADMM}
	\label{alg:PADMM-L}
	Given $(x_1^0,x_2^0,\cdots,x_N^0)\in(\mathcal{M}_1\cap X_1)\times(\mathcal{M}_2\cap X_2)\times\cdots \times(\mathcal{M}_{N-1}\cap X_{N-1})\times\RR^{n_N}$, $\lambda^0\in\RR^m$, $\beta>0$, $\gamma>0$, $H_i\succ 0, i=1,\ldots,N-1$.\\
	\For{$k = 0,1,... $}{
		$[\mbox{Step 1}]$ For $i = 1,2,...,N-1$ and positive semi-definite matrix $H_i$, compute
$x_i^{k+1} := \argmin_{x_i\in \mathcal{M}_i\cap X_i }\hat{\mathcal{L}}^i_{\beta}(x_i;x_1^{k+1},\cdots,x_{i-1}^{k+1},x_i^k,\cdots,x_N^k,\lambda^k)+\frac{1}{2}\|x_i-x_i^k\|^2_{H_i}$, \\
		$[\mbox{Step 2}]$ $x_{N}^{k+1} := x_N^k-\gamma\nabla_N\mathcal{L}_{\beta}(x_1^{k+1},\cdots,x_{N-1}^{k+1},x_N^k,\lambda^k)$,\\
		$[\mbox{Step 3}]$ $\lambda^{k+1} := \lambda^k-\beta(\sum_{i = 1}^{N}A_ix_i^{k+1}-b)$.
	}
\end{algorithm2e}

Essentially, instead of solving the subproblem involving the exact augmented Lagrangian defined by \eqref{Lagrangian}, we use the linearized approximation defined in \eqref{LagApprox}. It is also noted that the Steps 2 and 3 of Algorithm \ref{alg:PADMM-L} are the same as the ones in Algorithm \ref{alg:PADMM}, and thus Lemmas \ref{lm:PADMM-lemma1} and \ref{lm:PADMM-lemma3} still hold, as they do not depend on Step 1 of the algorithms. As a result, we only need to present the following lemma, which is a counterpart of Lemma \ref{lm:PADMM-lemma2}, and the proof is given in the appendix.


\begin{lemma}\label{lm:PADMM-L-lemma2}
Suppose that the sequence $(\xk_i,\cdots,\xk_N,\lambda^k)$ is generated by Algorithm \ref{alg:PADMM-L}. Let the parameters $\beta$ and $\gamma$ be defined according to \eqref{beta} and \eqref{gamma}, and $\Psi_G(x_1,\cdots,x_N,\lambda,\bar{x})$ be defined according to \eqref{Decrease_func}. If we choose $$H_i\succ\left(\frac{6L^2}{\beta}+L\right)I, ~for~i = 1,...,N-1,$$ then $\Psi_G(\xke_1,\cdots,\xke_N,\lambda^{k+1},\xk_N)$ monotonically decreases. More specifically, we have
\bea
\label{lm_:PADMM-L-lemma2:3}
& & \Psi_G(\xke_1,\cdots,\xke_{N-1},\xke_N,\lambda^{k+1},\xk_N) - \Psi_G(\xk_1,\cdots,\xk_{N-1},\xk_N,\lambda^{k},\xkm_N) \nonumber\\
&\leq& \left[\frac{\beta+L}{2}-\frac{1}{\gamma}+\frac{6}{\beta}\left(\beta-\frac{1}{\gamma}\right)^2+\frac{3L^2}{\beta}\right]\|\xk_N-\xke_N\|^2 \\
& & - \sum_{i = 1}^{N-1}\|\xk_i-\xke_i\|_{\frac{1}{2}H_i-\frac{L}{2}I-\frac{3L^2}{\beta}I}.\nonumber
\eea
Note that the right hand side of \eqref{lm_:PADMM-L-lemma2:3} is negative under the above conditions.
\end{lemma}

We are now ready to present the main complexity result for Algorithm \ref{alg:PADMM-L}, and the proof is omitted because it is very similar to that of Theorem \ref{thm:PADMM}.

\begin{theorem}\label{thm:PADMM-L}
Suppose the sequence $\{\xk_1,\cdots,\xk_N,\lambda^k\}$ is generated by Algorithm \ref{alg:PADMM-L}. Let the parameters $\beta$ and $\gamma$ satisfy \eqref{beta} and \eqref{gamma} respectively. Define $\kappa_1,\kappa_2,\kappa_3$ same as that in Theorem \ref{thm:PADMM}. Define
$$
\tau:= \min\left\{ -\left[\frac{\beta+L}{2}-\frac{1}{\gamma}+\frac{6}{\beta}\left(\beta-\frac{1}{\gamma}\right)^2+\frac{3L^2}{\beta}\right], \min_{i = 1,...,N-1}\left\{-\left(\frac{3L^2}{\beta}+\frac{L}{2}-\frac{\sigma_{\min}(H_i)}{2}\right) \right\} \right\}.
$$
Assume $H_i\succ\left(\frac{6L^2}{\beta}+L\right)I$, and let
$$
K = \left\lceil \frac{2\max\{\kappa_1,\kappa_2,\kappa_3\}}{\tau\epsilon^2}\left(\Psi_G(x_1^1,\cdots,x_N^1,\lambda^1,x_N^0)-\sum_{i=1}^{N-1}r_i^*-f^*\right)\right\rceil,
$$
and	$k^* = \argmin_{2\leq k\leq K+1}\sum_{i=1}^N(\|\xk_i-\xke_i\|^2+\|\xkm_i-\xk_i\|^2)$. Then, $(x_1^{k^*+1},\cdots,x_N^{k^*+1},\lambda^{k^*+1})$ is an $\epsilon$-stationary solution defined in Definition \ref{def:epsolu}.
\end{theorem}


\subsection{ Stochastic linearized proximal ADMM }

In machine learning applications, the objective is often in the form of
\[
f(x_1,\cdots,x_N) = \frac{1}{m}\sum_{i=1}^{m}f_i(x_1,\cdots,x_N),
\]
where $f_i$ corresponds to the loss function of the $i$th training data, and the sample size $m$ can be a very large number.
In rank-1 CP tensor decomposition problem, people aim to find the best rank-1 CP approximation of an order-$d$ tensor $\bT\in\mathbb{R}^{n_1\times\cdots\times n_d}$. With proper transformation, the objective function $f$ is
\[
f(x_1,\cdots,x_N) = \langle \bT,\otimes_{i=1}^dx_i\rangle,
\]
where complete description of $\bT$ is exponentially expensive.
In such cases, function evaluations in Algorithm \ref{alg:PADMM}, and the gradient evaluations in Algorithm \ref{alg:PADMM-L} are prohibitively expensive. In this section, we propose a nonconvex linearized stochastic proximal gradient-based ADMM with mini-batch to resolve this problem.
First, let us make the following assumption.

\begin{assumption}
For smooth $f$ and $i=1,\ldots,N$, there exists a stochastic first-order oracle that returns a noisy estimation to the partial gradient of $f$ with respect to $x_i$, and the noisy estimation $G_i(x_1,\cdots,x_N,\xi_i)$ satisfies
\begin{eqnarray}
	& & \E [G_i(x_1,\cdots,x_N,\xi_i)] = \nabla_i f(x_1,\cdots,x_N), \label{assumption-4-Unbias} \\
	& & \E \left[\|G_i(x_1,\cdots,x_N,\xi_i) - \nabla_i f(x_1,\cdots,x_N)\|^2\right]\leq \sigma^2, \label{assumption-5-VarBound}
\end{eqnarray}
where the expectation is taken with respect to the random variable $\xi_i$.
\end{assumption}

Let $M$ be the size of mini-batch, and denote
$$
G_i^M(x_1,\cdots,x_N) := \frac{1}{M}\sum_{j= 1}^M G_i(x_1,\cdots,x_N,\xi_i^j),
$$
where $\xi_i^j, j = 1,...,M$ are i.i.d.\ random variables. Clearly it holds that
\[\E [G_i^M(x_1,\cdots,x_N)] = \nabla_if(x_1,\cdots,x_N)\]
and
\bea
\label{BSFO}
\E \left[\|G_i^M(x_1,\cdots,x_N)- \nabla_if(x_1,\cdots,x_N)\|^2\right]\leq \sigma^2/M.
\eea
Now, the stochastic linear approximation of the augmented Lagrangian function with respect to block $x_i$ at point $(\hat{x}_1,\cdots,\hat{x}_N)$ is defined as (note that $r_N\equiv 0$):
\bea
\label{LagApprox-Stochastic}
\tilde{\cL}_{\beta}^i(x_i;\hat{x}_1,\cdots,\hat{x}_N,\lambda;M) & = &f(\hat{x}_1,\cdots,\hat{x}_N)+\langle G_i^M(\hat{x}_1,\cdots,\hat{x}_N),x_i-\hat{x}_i\rangle +r_i(x_i) \nonumber\\
& & -\bigg\langle \sum_{j\neq i}^{N}A_j\hat{x}_j+ A_ix_i-b,\lambda\bigg\rangle +\frac{\beta}{2}\bigg\|\sum_{j\neq i}^{N}A_j\hat{x}_j+ A_ix_i-b\bigg\|^2,
\eea
where $\lambda$ and $\beta>0$ follow the previous definitions. Compared to \eqref{LagApprox}, the full partial derivative $\nabla_if$ is replaced by the sample average of stochastic first-order oracles.

\begin{algorithm2e}[H]
	\caption{Nonconvex Linearized Stochastic Proximal Gradient-Based ADMM}
	\label{alg:PADMM-S}
	Given $(x_1^0,x_2^0,\cdots,x_N^0)\in(\mathcal{M}_1\cap X_1)\times(\mathcal{M}_2\cap X_2)\times \cdots \times(\mathcal{M}_{N-1}\cap X_{N-1})\times\RR^{n_N}$, $\lambda^0\in\RR^m$, $\beta>0$, $\gamma>0$, $H_i\succ 0, i=1,\ldots,N-1$, and the batch-size $M$. \\
	\For{$k = 0,1,...$ }{
		$\mbox{[Step 1]}$ For $i = 1,2,...,N-1$, and positive semi-definite matrix $H_i$, compute  $x_i^{k+1} = \argmin_{x_i\in \mathcal{M}_i\cap X_i }\tilde{\cL}_{\beta}^i(x_i;\xke_1,\cdots,\xke_{i-1},\xk_i,\cdots,\xk_N,\lambda^k;M)+\frac{1}{2}\|x_i-x_i^k\|^2_{H_i}$; \\
		$\mbox{[Step 2]}$ $x_{N}^{k+1} = x_N^k-\gamma\nabla_N\tilde{\mathcal{L}}^N_{\beta}(x_1^{k+1},\cdots,x_{N-1}^{k+1},x_N^k,\lambda^k)$; \\
		$\mbox{[Step 3]}$ $\lambda^{k+1} = \lambda^k-\beta(\sum_{i = 1}^{N}A_ix_i^{k+1}-b)$.
	}
\end{algorithm2e}


The convergence analysis of this algorithm follows the similar logic as that of the previous two algorithms. The proofs of these lemmas can be found in the appendix.

\begin{lemma}\label{lm:PADMM-S-lemma1} The following inequality holds:
	\bea
	\label{sto:1}
	\E[\|\lambda^{k+1}-\lambda^k\|^2] & \leq & 4\left(\beta-\frac{1}{\gamma}\right)^2\E[\|\xk_N-\xke_N\|^2]+4\left[\left(\beta-\frac{1}{\gamma}\right)^2+L^2\right]\E[\|\xkm_N-\xk_N\|^2] \nonumber \\
	& & +4L^2\sum_{i=1}^{N-1}\E[\|\xk_i-\xke_i\|^2] + \frac{8}{M}\sigma^2.
	\eea
\end{lemma}
In the stochastic setting, define the new potential function
\be
\label{sto:psi}
\Psi_S(x_1,\cdots,x_N,\lambda,\bar{x}) = \cL_\beta(x_1,\cdots,x_N,\lambda) + \frac{4}{\beta}\left[\left(\beta-\frac{1}{\gamma}\right)^2+L^2\right]\|\bar{x}-x_N\|^2.
\ee

\begin{lemma}\label{lm:PADMM-S-lemma2}
	Suppose the sequence $\{(x^k_1,..., x^k_N,\lambda_k)\}$ is generated by Algorithm \ref{alg:PADMM-S}. Define $\Delta = 17\beta^2-16(L+1)\beta-128L^2$, and assume that
	\be
	\label{sto:beta}
	\beta\in\left(\frac{8(L+1)+8\sqrt{(L+1)^2+34L^2}}{17},+\infty\right), H_i\succ\left(\frac{8L^2}{\beta}+L+1\right)I, \ i = 1,\ldots,N-1,
	\ee
	\be
	\label{sto:gamma}
	\gamma\in\left(\frac{16}{17\beta+\sqrt{\Delta}},\frac{16}{17\beta-\sqrt{\Delta}}\right).
	\ee
	Then it holds that
	\bea
	\label{sto:decrease_psi}
	& & \E[\Psi_S(\xke_1,\cdots,\xke_{N-1},\xke_N,\lambda^{k+1},\xk_N)] - \E[\Psi_S(\xk_1,\cdots,\xk_{N-1},\xk_N,\lambda^{k},\xkm_N)] \nonumber\\
	&\leq & \left[\frac{\beta+L}{2}-\frac{1}{\gamma}+\frac{8}{\beta}\left(\beta-\frac{1}{\gamma}\right)^2+ \frac{4L^2}{\beta}+\half\right]\E[\|\xke_N-\xk_N\|^2] \\
	& & - \sum_{i = 1}^{N-1}\E\left[\|\xk_i-\xke_i\|^2_{\frac{1}{2}H_i-\frac{4L^2}{\beta}I-\frac{L+1}{2}I}\right]+ \left(\frac{8}{\beta}+\frac{N}{2}\right)\frac{\sigma^2}{M},\nonumber
	\eea
and the coefficient in front of $\E[\|\xke_N-\xk_N\|^2]$ is negative.
\end{lemma}

\begin{lemma}\label{lm:PADMM-S-lemma3}
	Suppose the sequence $\{\xk_1,\cdots,\xk_N,\lambda^k\}$ is generated by Algorithm \ref{alg:PADMM-S}. It holds that
	\bea
	\label{sto:lowerbound}
	\E[\Psi_S(\xke_1,\cdots,\xke_N,\lambda^{k+1},\xk_N)] &\geq& \sum_{i=1}^{N-1}r_i^*+f^*-\frac{2\sigma^2}{\beta M} \geq \sum_{i=1}^{N-1}r_i^*+f^*-\frac{2\sigma^2}{\beta}.
	\eea
\end{lemma}

We are now ready to present the iteration complexity result for Algorithm \ref{alg:PADMM-S}.

\begin{theorem}\label{thm:PADMM-S}
	Suppose that the sequence $\{\xk_1,...,\xk_N,\lambda^k\}$ is generated by Algorithm \ref{alg:PADMM-S}. Let the parameters $\beta$ and $\gamma$ satisfy \eqref{sto:beta} and \eqref{sto:gamma} respectively. Define $\kappa_1 := \frac{4}{\beta^2}\left[\left(\beta-\frac{1}{\gamma}\right)^2+L^2\right]$, $\kappa_2 := 3\left[\left(\beta-\frac{1}{\gamma}\right)^2+L^2\right]$, $\kappa_3 := 2\left(L+\beta\sqrt{N}\max_{1\leq i\leq N}\{\|A_i\|_2^2\}+\max_{1\leq i\leq N-1}\|H_i\|_2\right)^2$ , $\kappa_4 = \frac{2}{\tau}\left(\frac{8}{\beta}+\frac{N}{2}\right)$ with \\
	$\tau:= \min\left\{ -\left(\frac{\beta+L}{2}-\frac{1}{\gamma}+\frac{8}{\beta}\left[\beta-\frac{1}{\gamma}\right]^2+\frac{4L^2}{\beta}+\frac{1}{2}\right), \min_{i = 1,...,N-1}\left\{-\left(\frac{4L^2}{\beta}+\frac{L+1}{2}-\frac{\sigma_{\min}(H_i)}{2}\right) \right\} \right\} $ . Assume   $H_i\succ(\frac{8L^2}{\beta}+L+1)I$ and let  $$M\geq\frac{2\sigma^2}{\epsilon^2}\max\{\kappa_1\kappa_4+\frac{8}{\beta^2},\kappa_2\kappa_4+3,\kappa_3\kappa_4+2\},$$
	$$K = \left\lceil \frac{4\max\{\kappa_1,\kappa_2,\kappa_3\}}{\tau\epsilon^2}\left(\E[\Psi_G(x_1^1,...,x_N^1,\lambda^1,x_N^0)]-\sum_{i=1}^{N-1}r_i^*-f^*+\frac{2\sigma^2}{\beta}\right)\right\rceil.$$
	Let $k^* = \argmin_{2\leq k\leq K+1}\sum_{i=1}^N(\|\xk_i-\xke_i\|^2+\|\xkm_i-\xk_i\|^2),$ then $(x_1^{k^*+1},\cdots,x_N^{k^*+1},\lambda^{k^*+1})$is an $\epsilon$-stationary solution in accordance of Definition \ref{def:epsolu-exp}.
\end{theorem}

\begin{proof}
Most parts of the proof are similar to that of Theorem \ref{thm:PADMM}, the only difference is that we need to carry the stochastic errors throughout the process. For simplicity, we shall highlight the key differences.
First, we define $\theta_k$ according to \eqref{thm2_1} and then bound $\E[\theta_{k^*}]$ by
\bea
\label{sto:thm1}
\E[\theta_{k^*}] & \leq & \min_{k=2,...,K+1}\E[\theta_k] \\
& \leq & \frac{2}{\tau K}\left(\E[\Psi_S(x_1^1,...,x_N^1,\lambda^1,x_N^0)]-\sum_{i=1}^{N-1}r_i^*-f^*+\frac{2\sigma^2}{\beta}\right) + \kappa_4\frac{\sigma^2}{M}.\nonumber
\eea
Second, we have
\be
\label{sto:thm2}
\E \left[\|\lambda^{k+1}-\nabla_Nf(\xke_1,...,\xke_N)\|^2\right] \leq \kappa_2\E[\theta_k] + \frac{3\sigma^2}{M},
\ee
\be
\label{sto:thm3}
\E \left[\|\sum_{i=1}^{N-1}A_i\xke_i+\xke_N-b\|^2\right]\leq \kappa_1\E[\theta_k]+\frac{8\sigma^2}{\beta^2M},
\ee
and
\bea \E\left[\dist\left(\Proj_{\cT_{x_i^{k+1}}\mathcal{M}_i}\left(-\nabla_if(x^{k+1})+A_i^\top \lambda^{k+1}-\partial r_i(x_i^{k+1})\right),\mathcal{N}_{X_i}(x_i^{k+1})\right)^2\right]
\leq \kappa_3\E[\theta_k]+\frac{2\sigma^2}{M} . \label{sto:otherblock}
\eea
Finally, apply Jensen's inequality $\E_{\xi}[\sqrt{\xi}]\leq \sqrt{\E_{\xi}[\xi]}$ to the above bounds \eqref{sto:thm1}, \eqref{sto:thm2} and \eqref{sto:thm3}, and choose $K$ as defined, the $\epsilon$-stationary solution defined in \eqref{def:epsolu-exp} holds in expectation.
\end{proof}

\subsection{ A feasible curvilinear line-search variant of ADMM }
We remark that the efficacy of the previous algorithms rely on the solvability of 
the subproblems at Step 1. Though the subproblems may be easy computable as we shall see from application examples in Section \ref{sec:application},
there are also examples where such solutions are not available for many manifolds even when the objective is linearized.
As a remedy
we present in this subsection a feasible curvilinear line-search based variant of the ADMM. First let us make a few additional assumptions.

\begin{assumption}
	In problem \eqref{prob:main}, the manifolds $\cM_i,i = 1,\ldots,N-1$ are compact. The nonsmooth regularizing functions $r_i(x_i)$ vanish, and the constraint sets $X_i = \R^{n_i}$, for $i = 1,...,N-1$.
\end{assumption}

Accordingly, the third part of the optimality condition \eqref{opt_ADMM} is simplified to
\be
\label{opt_ADMM_retraction}
\Proj_{\cT_{x_i^*}\mathcal{M}_i}\left(\nabla_i f(x^*)-A_i^\T \lambda^* \right) = 0, i = 1,...,N-1.
\ee


Let  $R_i(\bar{x}_i, t g)$ be a retraction operator at point $\bar{x}_i\in\mathcal{M}_i$ in direction $g\in\mathcal{T}_{\bar{x}_i}\mathcal{M}_i$. Then a parameterized curve $Y_i(t) = R_i(\bar{x}_i, t g)$ is defined on $\mathcal{M}_i$. In particular, it satisfies
\be
\label{retraction}
Y_i(0) = \bar{x}_i \mbox{ and }  Y'_i(0) = g.
\ee
\begin{proposition}
	\label{assumption-6-retration}
	For retractions $Y_i(t) = R_i(\bar{x}_i, t g), i = 1,...,N-1$, there exist $ L_1,L_2>0$ such that
	\bea
	\|Y_i(t) - Y_i(0)\| & \leq & L_1t\|Y_i'(0)\|,\\
	\|Y_i(t) - Y_i(0) - tY_i'(0)\| & \leq &L_2t^2\|Y_i'(0)\|^2.	
	\eea
\end{proposition}
	
The above proposition states that the retraction curve is approximately close to a line in Euclidean space. It was proved as a byproduct of Lemma 3 in \cite{RM_glo} and was also adopted by \cite{Jiang-SVRG-RM-2017}. Let the augmented Lagrangian function be defined by \eqref{Lagrangian} (without the $r_i(x_i)$ terms) and denote
$$\grad_{x_i} \mathcal{L}_{\beta}(x_1,\cdots,x_N,\lambda) = \Proj_{\mathcal{T}_{x_i}\mathcal{M}_i}\big\{ \nabla_i\mathcal{L}_{\beta}(x_1,\cdots,x_N,\lambda)\big\}$$
as the Riemannian partial gradient. We present the algorithm as in Algorithm \ref{alg:retraction}.

\begin{algorithm2e}[H]
	\caption{A feasible curvilinear line-search-based ADMM}
	\label{alg:retraction}
	Given $(x_1^0,\cdots,x_{N-1}^0,x_N^0)\in\mathcal{M}_1 \times\cdots \times\mathcal{M}_{N-1}\times\RR^{n_N}$, $\lambda^0\in \RR^m$, $\beta,\gamma,\sigma>0, s>0$ , $\alpha\in(0,1)$. \\
	\For{$k = 0,1,...$ }{
		$[\mbox{Step 1}]$ \For {$i = 1,2,...,N-1$}{
			Compute $g_i^k = \grad_{x_i} \mathcal{L}_{\beta}(x_1^{k+1},\cdots,x_{i-1}^{k+1},x_{i}^{k},\cdots,x_N^k,\lambda^k)$; \\
			Initialize with $t_i^k = s$. While
			\bea
			& & \mathcal{L}_{\beta}(x_1^{k+1},\cdots,x_{i-1}^{k+1}, R_i(x_i^k,-t_i^kg_i^k),x_{i+1}^{k},\cdots,x_N^k,\lambda^k)\nonumber\\
			& > & \mathcal{L}_{\beta}(x_1^{k+1},\cdots,x_{i-1}^{k+1},x_{i}^{k},\cdots,x_N^k,\lambda^k) - \frac{\sigma}{2}(t_i^k)^2\|g_i^k\|^2, \nonumber
			\eea
			shrink $t_i^k$ by $t_i^k \leftarrow \alpha t_i^k$;\\
			Set $x_i^{k+1} = R_i(x_i^k,-t_i^kg_i^k);$
		}
		$[\mbox{Step 2}]$ $x_{N}^{k+1} := x_N^k-\gamma\nabla_N\mathcal{L}_{\beta}(x_1^{k+1},\cdots,x_{N-1}^{k+1},x_N^k,\lambda^k)$; \\
		$[\mbox{Step 3}]$ $\lambda^{k+1} := \lambda^k-\beta(\sum_{i = 1}^{N}A_ix_i^{k+1}-b)$.
	}
\end{algorithm2e}

For Steps 2 and 3, Lemma \ref{lm:PADMM-lemma1} and Lemma \ref{lm:PADMM-lemma3} still hold. Further using Proposition \ref{assumption-6-retration}, Lemma~\ref{lm:PADMM-lemma1} becomes
\begin{lemma}
	\label{lm:retration-lemma1}
	Suppose that the sequence $\{x_1^k,...,\xk_N,\lambda^k\}$ is generated by Algorithm \ref{alg:retraction}. Then,
	\bea
	\|\lambda^{k+1}-\lambda^k\|^2 & \leq & 3\left(\beta-\frac{1}{\gamma}\right)^2\|t_N^kg_N^k\|^2+3\left[\left(\beta-\frac{1}{\gamma}\right)^2+L^2\right]\|t_N^{k-1}g_N^{k-1}\|^2 \nonumber \\
	& & +3L^2L_1^2\sum_{i=1}^{N-1}\|t_i^kg_i^k\|^2,
	\eea
\end{lemma}
where we define $t_N^k = \gamma$ and $\xke_N = \xk_N+t_N^kg_N^k, \forall k \geq0,$ for simplicity. Moreover, for the definition of $\Psi_G$ in \eqref{Decrease_func}, Lemma \ref{lm:PADMM-lemma2} remains true, whereas the amount of decrease becomes
\bea
\label{decrease-retraction}
& & \Psi_G(\xke_1,\cdots,\xke_{N-1},\xke_N,\lambda^{k+1},\xk_N) - \Psi_G(\xk_1,\cdots,\xk_{N-1},\xk_N,\lambda^{k},\xkm_N) \\
&\leq& \left[\frac{\beta+L}{2}-\frac{1}{\gamma}+\frac{6}{\beta}\left(\beta-\frac{1}{\gamma}\right)^2+\frac{3L^2}{\beta}\right]\|t_N^kg_N^k\|^2 - \sum_{i = 1}^{N-1}\left(\frac{\sigma}{2}-\frac{3}{\beta}L^2L_1^2\right)\|t_i^kg_i^k\|^2< 0.\nonumber
\eea
Now we are in a position to present the iteration complexity result, where the detailed proof can be found in the appendix.
\begin{theorem}\label{thm:retraction}
	Suppose that the sequence $\{\xk_1,...,\xk_N,\lambda^k\}$ is generated by Algorithm \ref{alg:retraction}, and the parameters $\beta$ and $\gamma$ satisfy \eqref{beta} and \eqref{gamma} respectively. Denote $A_{\max} = \max_{1\leq j\leq N}\|A_j\|_2$. Define \\$\tau:= \min\left\{ -\left[\frac{\beta+L}{2}-\frac{1}{\gamma}+\frac{6}{\beta}\left(\beta-\frac{1}{\gamma}\right)^2+\frac{3L^2}{\beta}\right], \frac{\sigma}{2}-\frac{3}{\beta}L^2L_1^2 \right\} $, $\kappa_1 := \frac{3}{\beta^2}\left[\left(\beta-\frac{1}{\gamma}\right)^2+L^2\cdot\max\{L_1^2,1\}\right]$, $\kappa_2 := \left(|\beta-\frac{1}{\gamma}|+L\right)^2$, $\kappa_3 := \left((L+\sqrt{N}\beta A_{\max}^2)\cdot\max\{L_1,1\} +\frac{\sigma+2L_2C+(L+\beta A_{\max}^2)L_1^2}{2\alpha} + \beta A_{\max}\sqrt{\kappa_1}\right)^2$, where $C>0$ is a constant that depends only on the first iterate and the initial point. Assume $\sigma>\max\{\frac{6}{\beta}L^2L_1^2,\frac{2\alpha}{s}\}$. Define
	\be\label{def:K}
	K := \left\lceil\frac{3\max\{\kappa_1,\kappa_2,\kappa_3\}}{\tau\epsilon^2}
	\left(\Psi_G(x_1^1,...,x_N^1,\lambda^1,x_N^0)-f^*\right)\right\rceil,
	\ee
	and $k^* := \argmin_{2\leq k\leq K+1}\sum_{i=1}^N(\|t_i^{k+1}g_i^{k+1}\|^2+\|t_i^kg_i^k\|^2+\|t_i^{k-1}g_i^{k-1}\|^2).$ Then $(x_1^{k^*+1},\cdots,x_N^{k^*+1},\lambda^{k^*+1})$ is an $\epsilon$-stationary solution of \eqref{prob:main}.
\end{theorem}

\section{ Extending the Basic Model }\label{sec:Extension}
Recall that for our basic model \eqref{prob:main}, a number of assumptions have been made; e.g.\ we assumed that $r_i,i=1,...,N-1$ are convex, $x_N$ is unconstrained and $A_N = I$. In this section we shall extend the model to relax these assumptions. We shall also extend our basic algorithmic model from the Gauss-Seidel updating style to allow the Jacobi style updating, to enable parallelization.

\subsection{Relaxing the convexity requirement on nonsmooth regularizers}

For problem \eqref{prob:main} the nonsmooth part $r_i$ are actually not necessarily convex. As an example, nonconvex and nonsmooth regularizations such as $\ell_q$ regularization with $0<q<1$ are very common in compressive sensing.
To accommodate the change, the following adaptation is needed.

\begin{proposition}
For problem \eqref{prob:main}, where $f$ is smooth with Lipchitz continuous gradient. Suppose that $\cI_1,\cI_2$ form a partition of the index set $\{1,...,N-1\}$, in such a way that for $i\in\cI_1$, $r_i$'s are nonsmooth but convex, and for $i\in\cI_2$, $r_i$'s are nonsmooth and nonconvex but are locally Lipschitz continuous. If for blocks $x_i,i\in\cI_2$ there are no manifold constraints, i.e.\ $\cM_i = \RR^{n_i},i\in\cI_2$, then Theorems \ref{thm:PADMM}, \ref{thm:PADMM-L} and \ref{thm:PADMM-S} remain true.
\end{proposition}

Recall that in the proofs for \eqref{subPopt} and \eqref{otherblock}, we required the convexity of $r_i$ to ensure \eqref{proj_sub}. However, if $\cM_i = \RR^{n_i}$, then we directly have \eqref{Clarke:f+r}, i.e., $\partial_i(f+r_i) = \nabla_if+\partial r_i$
instead of \eqref{proj_sub}. The only difference is that $\partial r_i$ becomes the Clarke generalized subdifferential instead of the convex subgradient and the projection operator is no longer needed. In the subsequent complexity analysis, we just need to remove all the projection operators in \eqref{otherblock} and \eqref{sto:otherblock}. Hence the same convergence result follows.

Moreover, if for some blocks, $r_i$'s are nonsmooth and nonconvex, while the constraint $x_i\in\cM_i\neq\RR^{n_i}$ is still imposed, then we can solve the problem via the following equivalent formulation:
\bea
\label{prob:reformulate:Extn-1}
& \min & f(x_1,...,x_N) + \sum_{i \in \cI_1\cup\cI_2} r_i(x_i)+ \sum_{i \in \cI_3} r_i(y_i)\nonumber\\
& \st & \sum_{i = 1}^{N} A_ix_i = b, \mbox{ with } A_N = I,  \nonumber \\
& & x_N\in \RR^{n_N},\\
& & x_i \in \mathcal{M}_i\cap X_i, ~~ i\in\cI_1\cup\cI_3, \nonumber \\
& & x_i \in X_i, ~~ i \in\cI_2, \nonumber\\
& & y_i = x_i, ~~ i \in\cI_3, \nonumber
\eea
where $\cI_1,\cI_2$ and $\cI_3$ form a partition for $\{1,...,N-1\}$, with $r_i$ convex for $i\in\cI_1$ and nonconvex but locally Lipschitz continuous for $i\in\cI_2\cup\cI_3$. The difference is that $x_i$ is not required to satisfy Riemannian manifold constraint for $i\in\cI_2$. 

Unfortunately, the $\ell_q$ regularization itself is not locally Lipschitz at $0$ and hence does not satisfy our requirement. But if we apply the modification of $\ell_q$ regularization in Remark \ref{rm:Lq}, then we can circumvent this difficulty while making almost no change to the solution process and keeping closed form solutions. In fact, due to the limited machine precision of computer, we can directly use $\ell_q$ regularization and treat it as if we were working with the modified $\ell_q$ regularization.

\subsection{Relaxing the condition on the last block variables}
In the previous discussion, we limit our problem to the case where $A_N = I$ and $x_N$ is unconstrained. Actually, for the  general case
\begin{eqnarray}
\label{prob:general-ADMM}
& \min & f(x_1,\cdots,x_N) + \sum_{i = 1}^{N} r_i(x_i) \nonumber\\
& \st & \sum_{i = 1}^{N} A_ix_i = b,  \\
& & x_i \in \mathcal{M}_i\cap X_i, ~~ i = 1,...,N, \nonumber
\end{eqnarray}
where $x_N$ is as normal as other blocks, we can actually add an additional block $x_{N+1}$ and modify the objective a little bit and arrive at the modified problem
\begin{eqnarray}
\label{prob:modified-ADMM}
& \min & f(x_1,\cdots,x_N,x_{N+1}) + \sum_{i = 1}^{N} r_i(x_i) + \frac{\mu}{2}\|x_{{N+1}}\|^2 \nonumber\\
& \st & \sum_{i = 1}^{N} A_ix_i +x_{N+1}= b, \nonumber \\
& & x_{N+1}\in \RR^m, \\
& & x_i \in \mathcal{M}_i\cap X_i, ~~ i = 1,...,N. \nonumber
\end{eqnarray}
Following a similar line of proofs of Theorem 4.1 in \cite{NcvxADMM:Zhang-etal-2016}, we have the following proposition.
\begin{proposition} Consider the modified problem \eqref{prob:modified-ADMM} with $\mu = 1/\epsilon$ for some given tolerance $\epsilon\in(0,1)$ and suppose the sequence $\{(x_1^k,...,x_{N+1}^k,\lambda^k)\}$ is generated by Algorithm \ref{alg:PADMM} (resp.\ Algorithm \ref{alg:PADMM-L}). Let  $(x_1^{k*+1},...,x_{N}^{k*+1},\lambda^{k*+1})$ be $\epsilon$-stationary solution of \eqref{prob:modified-ADMM} as defined in Theorem \ref{thm:PADMM} (resp.\ Theorem \ref{thm:PADMM-L}). Then  $(x_1^{k*+1},...,x_{N}^{k*+1},\lambda^{k*+1})$ is an $\epsilon$-stationary point of the original problem \eqref{prob:general-ADMM}. 
\end{proposition}

\begin{remark}
We remark here that when $\mu = 1/\epsilon$, the Lipschitz constant of the objective function $L$ also depends on $\epsilon$. As a result, the iteration complexity of Algorithms \ref{alg:PADMM} and \ref{alg:PADMM-L} becomes $O(1/\epsilon^4)$.
\end{remark}


\subsection{The Jacobi-style updating rule}
Parallel to \eqref{LagApprox}, we define a new linearized approximation of the augmented Lagrangian as
\bea
\label{LagApprox2}
\bar{\cL}_{\beta}^i(x_i;\hat{x}_1,\cdots,\hat{x}_N,\lambda) & = &\bar{f}_{\beta}(\hat{x}_1,\cdots,\hat{x}_N)+\langle \nabla_i\bar{f}_{\beta}(\hat{x}_1,\cdots,\hat{x}_N),x_i-\hat{x}_i\rangle \nonumber\\
& & -\bigg\langle \sum_{j\neq i}^{N}A_j \hat{x}_j + A_ix_i-b,\lambda\bigg\rangle +r_i(x_i),
\eea
where
$$\bar{f}_{\beta}(x_1,\cdots,x_N) = f(x_1,\cdots,x_N)+ \frac{\beta}{2}\bigg\|\sum_{j=1}^{N}A_jx_j-b\bigg\|^2.$$

Compared with \eqref{LagApprox}, in this case we linearize both the coupling smooth objective function and the augmented term.

In Step 1 of  Algorithm \ref{alg:PADMM-L}, we have the Gauss-Seidel style updating rule,
$$x_i^{k+1} = \argmin_{x_i\in \mathcal{M}_i\cap X_i }\hat{\mathcal{L}}^i_{\beta}(x_i;x_1^{k+1},\cdots,x_{i-1}^{k+1},x_i^k,\cdots,x_N^k,\lambda^k)+\frac{1}{2}\|x_i-x_i^k\|^2_{H_i}.$$
Now if we replace this with the Jacobi style updating rule,
\be\label{Jacobi-subproblems}
x_i^{k+1} = \argmin_{x_i\in \mathcal{M}_i\cap X_i }\bar{\mathcal{L}}^i_{\beta}(x_i;x_1^{k},\cdots,x_{i-1}^{k},x_i^k,\cdots,x_N^k,\lambda^k)+\frac{1}{2}\|x_i-x_i^k\|^2_{H_i},
\ee
then we end up with a new algorithm which updates all blocks parallelly instead of sequentially. When the number of blocks, namely $N$, is large, using the Jacobi updating rule can be beneficial because the computation can be parallelized.

To establish the convergence of this process, all we need is to establish a counterpart of \eqref{lm_:PADMM-L-lemma2:1} in this new setting, namely
\be
\label{Jacobi}
\cL_{\beta}(\xke_1,\cdots,\xke_{N-1},\xk_N,\lambda^k)\leq\cL_{\beta}(\xk_1,\cdots,\xk_N,\lambda^k)-\sum_{i=1}^{N-1}\|\xk_i-\xke_i\|^2_{\frac{H_i}{2}-\frac{\hat{L}}{2}I},
\ee
for some $\hat{L}>0.$ Consequently, if we choose
$H_i\succ\hat{L}I,$
then the convergence and complexity analysis goes through for Algorithm \ref{alg:PADMM-L}. Moreover, Algorithm \ref{alg:PADMM-S} can also be adapted to the Jacobi-style updates. The proof for \eqref{Jacobi} is given in the appendix.

\section{Some Applications and Their Implementations }
\label{sec:application}


The applications of block optimization with manifold constraints are abundant. In this section we shall present some typical examples. Our choices include the NP-hard maximum bisection problem, the sparse multilinear principal component analysis, and the community detection problem.

\subsection{ Maximum bisection problem } The maximum bisection problem is a variant of the well known NP-hard maximum cut problem. Suppose we have a graph $G = (V,E)$ where $V = \{1,...,n\} := [n]$ denotes the set of nodes and $E$ denotes the set of edges, each edge $e_{ij}\in E$ is assigned with a weight $W_{ij}\geq0$. For pair $(i,j)\notin E$, define $W_{ij} = 0$. Let a bisection $\{V_1,V_2\}$ of $V$ be defined as
$$
V_1\cup V_2 = V,\quad V_1\cap V_2 = \emptyset, \quad |V_1|=|V_2|.
$$
The maximum bisection problem is to find the best bisection that maximize the graph cut value:
\bea
&\max_{V_1,V_2} & \sum_{i \in V_1}\sum_{j \in V_2} W_{ij} \nonumber\\
& \text{s.t.} & V_1, V_2 \mbox{ is a bisection of } V. \nonumber
\eea
Note that if we relax the constraint $|V_1| = |V_2|$, that is, we only require $\{V_1,V_2\}$ to be a partition of $V$, then this problem becomes the \emph{maximum cut} problem. In this paper, we propose to solve this problem by our method and compare our results with the two SDP relaxations proposed in \cite{Max-bisec:Frieze-1997,Max-bisec:Ye-2011}.

First, we model the bisection $\{V_1,V_2\}$ by a binary assignment matrix $U\in\{0,1\}^{n\times2}$. Each node $i$ is represented by the $i$th row of matrix $U$. Denote this row by $u^\top_i$, where $u_i\in\{0,1\}^{2\times1}$ is a column vector with exactly one entry equal to 1. Then $u_i^\top = (1,0)$ or $(0,1)$ corresponds to $i\in V_1$ or $i\in V_2$ respectively, and the objective can be represented by
$$\sum_{i \in V_1}\sum_{j \in V_2} W_{ij} = \sum_{i,j}(1-\langle u_i,u_j\rangle)W_{i,j} = -\langle W, UU^\top \rangle + const.$$
The constraint that $|V_1| = |V_2|$ is characterized by the linear equality constraint $$\sum_{i=1}^n (u_i)_1 - \sum_{i=1}^n (u_i)_2 = 0.$$
Consequently, we can develop the nonconvex relaxation of the maximum bisection problem as
\bea
\label{prob:max-bisec-0}
& \min_{U} & \langle W,UU^\top\rangle \nonumber\\
& \text{s.t.} & \|u_i\|^2 = 1, u_i\geq 0, \mbox{ for } i = 1,...,n, \\
& & \sum_{i=1}^n (u_i)_1 - \sum_{i=1}^n (u_i)_2 = 0. \nonumber
\eea

After the relaxation is solved, each row is first rounded to an integer solution
\[
u_i \leftarrow \begin{cases}
(1,0)^\top, &\mbox{if } (u_i)_1\geq (u_i)_2,\\
(0,1)^\top, &\mbox{otherwise.}
\end{cases}
\]
Then a greedy algorithm is applied to adjust current solution to a feasible bisection solution. Note that this greedy step is necessary for our algorithm and the SDP relaxations in \cite{Max-bisec:Frieze-1997,Max-bisec:Ye-2011} to reach a feasible bisection.

The ADMM formulation of this problem will be shown in the numerical experiment part and the algorithm realization is omitted. Here we only need to mention that all the subproblems are of the following form:
\bea
\label{prob:shpere}
&\min_x & b^\T x \\
& \st & \|x\|^2=1, x\geq0. \nonumber
\eea

This nonconvex constrained problem can actually be solved to global optimality in closed form, see the Lemma 1 in \cite{CP-comm}. For the sake of completeness, we present the lemma bellow.
\begin{lemma}
	\label{prob:sub-closed-form}
	(Lemma 1 in \cite{CP-comm}.) Define $b^+ = \max\{b,0\}$, $b^- = -\min\{b,0\}$, where $\max$ and $\min$ are taken element-wise. Note that $b^+\geq 0$, $b^-\geq 0$, and $b = b^+-b^-$. The closed form solution for problem \eqref{prob:shpere} is
	\be x^* = \begin{cases} \frac{b^-}{\|b^-\|}, & \mbox{ if } b^-\neq0
		\\
		e_i,  & \mbox{ otherwise},
	\end{cases}\ee
	where $e_i$ is the $i$-th unit vector with $i = \argmin_j b_j$.
\end{lemma}



\subsection{The $\ell_q$-regularized sparse tensor PCA}

As we discussed at the beginning of Section~\ref{introduction}, the tensor principal component analysis (or multilinear principal component analysis (MPCA)) has been a popular subject of study in recent years. Below, we shall discuss a sparse version of this problem.

Suppose that we are given a collection of order-$d$ tensors $\bT^{(1)},\bT^{(2)},...,\bT^{(N)}\in\RR^{n_1\times n_2\times\cdots\times n_d}$. 
The sparse MPCA problem can be formulated as (see also \cite{YYX}):
\begin{eqnarray*}
& \min & \sum_{i =1}^{N}\|\bT^{(i)}-\bC^{(i)}\times_1U_1\times\cdots\times_dU_d\|_F^2 + \alpha_1\sum_{i=1}^{N}\|\bC^{(i)}\|_p^p+\alpha_2\sum_{j=1}^{d}\|U_j\|_q^q\nonumber\\
& \st & \bC^{(i)}\in\RR^{m_1\times\cdots\times m_d}, i = 1,...,N  \\
&   & U_j\in\RR^{n_j\times m_j}, U_j^\T U_j = I, j = 1,...,d.
\end{eqnarray*}


In order to apply our developed algorithms, we can consider the following variant of sparse MPCA:
\begin{equation} \label{prob:MPCA}
\begin{array}{ll}
\min & \sum_{i=1}^{N} \|\bT^{(i)}-\bC^{(i)}\times_1U_1\times\cdots\times_dU_d\|_F^2 + \alpha_1\sum_{i=1}^{N}\|\bC^{(i)}\|_p^p+\alpha_2\sum_{j=1}^{d}\|V_j\|_q^q+\frac{\mu}{2}\sum_{j=1}^{d}\|Y_j\|^2 \\
\st  & \bC^{(i)}\in\RR^{m_1\times\cdots\times m_d}, i = 1,...,N  \\
     & U_j\in\RR^{n_j\times m_j}, U_j^\T U_j = I, j = 1,...,d \\
     & V_j - U_j+Y_j=0, j = 1,...,d.
\end{array}
\end{equation}
Note that this model is different from the ones used in \cite{SMPCA1,SMPCA2}.

Denote $\bT_{(j)}^{(i)}$ to be the mode-$j$ unfolding of a tensor $\bT^{(i)}$, and denote $\bC$ to be the set of all tensors $\{\bC^{(i)}: i = 1,...,N\}$. 
The augmented Lagrangian function of \eqref{prob:MPCA} is
\begin{eqnarray*}
	L_{\beta}(\bC, U, V, Y, \Lambda) & = & \sum_{i =1}^{N} \|\bT^{(i)}-\bC^{(i)}\times_1U_1\times\cdots\times_dU_d\|_F^2 + \alpha_1\sum_{i=1}^{N}\|\bC^{(i)}\|_p^p+\alpha_2\sum_{j=1}^{d}\|V_j\|_q^q \\
	& &+\frac{\mu}{2}\sum_{j=1}^{d}\|Y_j\|^2 -\sum_{j=1}^{d}\langle U_j-V_j+Y_j, \Lambda_j\rangle +\frac{\beta}{2}\sum_{j=1}^{d}\| U_j-V_j+Y_j\|_F^2 .
\end{eqnarray*}

An implementation of the Algorithm \ref{alg:PADMM} for solving \eqref{prob:MPCA} is shown in Algorithm \ref{alg:mpca}.

\begin{algorithm2e}
	\caption{A typical iteration of Algorithm \ref{alg:PADMM} for solving \eqref{prob:MPCA}}\label{alg:mpca}
	[Step 1] 	\For{$j = 1,...,d$ }{
		Set $B = \sum_{i=1}^{N}\bT_{(j)}^{(i)}(U_d\otimes\cdots\otimes U_{j+1}\otimes U_{j-1}\otimes\cdots\otimes U_1)(\bC_{(j)}^{(i)})^\T 
+\frac{1}{2}\Lambda_j-\frac{\beta}{2}Y_j+\frac{\beta}{2}V_j+\frac{\sigma}{2}U_j$ \\
		$U_j\leftarrow\argmin_{U^\T U=I} -\langle2B,U\rangle$
	}
	
	[Step 2] 	\For{$j = 1,...,d$ }{
		For each component $V_j(s)$ where $s = (s_1,s_2)$ is a multilinear index, \\
		set $b = \beta Y_j(s)+\beta U_j(s)-\Lambda_j(s)+\sigma V_j(s).$ \\
		$V_j(s) = \argmin_x \frac{\beta+\sigma}{2}x^2+\alpha_2|x|^q-bx$
	}
	
	[Step 3] \For{$i = 1,...,N$ }{
		For each component $\bC^{(i)}(s)$, where $s = (s_1,...,s_d)$ is a multilinear index, \\
		set $b = \sigma\bC^{(i)}(s)-2\left[(U_d^\T \otimes\cdots\otimes U_1^\T )\mathrm{vec}(\bT^{(i)})\right](s)$.
		$\bC^{(i)}(s)\leftarrow \argmin_x \frac{2+\sigma}{2}x^2+\alpha_1|x|^q-bx$
	}
	[Step 4] \For{$j = 1,...,d$ }{
		$Y_j\leftarrow Y_j - \eta\left[(\beta+\mu)Y_j-\beta U_j-\beta V_j -\Lambda_j\right]$
	}
	[Step 5] \For{$j = 1,...,d$ }{
		$\Lambda_j\leftarrow \Lambda_j - \beta\left(U_j-V_j+Y_j\right)$
	}
\end{algorithm2e}

In Step 1 of Algorithm \ref{alg:mpca}, the subproblem to be solved is
\be
\label{subprob:stiefel}
U_j = \argmin_{U^\T U=I} -\langle2B,U\rangle = \argmin_{U^\T U=I} \|B-U\|_F^2,
\ee
which is known as the nearest orthogonal matrix problem. Suppose we have the SVD decomposition of the matrix B as $B = Q\Sigma P^\T$, then the global optimal solution is $U_j = QP^\T$. When $B$ has full column rank, the solution is also unique.


In Steps 2 and 3 of Algorithm \ref{alg:mpca}, they are actually a group of one-dimensional decoupled problems. Since no nonnegative constraints are imposed, we can apply $\ell_1$ regularization for which soft-thresholding gives closed form solution to the subproblems. However, if we want to apply $\ell_q$ refularization for $0<q<1$, then the subproblem amounts to solve
\bea
\label{prob:L_q}
\min f(x) = ax^2+bx+c|x|^q,
\eea
where $0<q<1$, $a>0$, $c>0$. The function is nonconvex and nonsmooth at $0$ with $f(0) = 0$. For $x>0$, we can take the derivative and set it to 0, and obtain $2ax+qcx^{q-1}+b = 0,$ or equivalently
$$2ax^{2-q}+bx^{1-q}+cq = 0.$$
If $q = \frac{1}{2}$, then setting $z = \sqrt{x}$ leads to
$2az^3+bz+cq = 0.$
If $q = \frac{2}{3}$, then setting $z = x^{\frac{1}{3}}$ leads to
$2az^4+bz+cq = 0.$ In both cases, we have closed-form solutions. Similarly, we apply this trick to the case when $x<0$. Suppose we find the roots $x_1,...,x_k$ and we set $x_0 = 0$, then the solution to \eqref{prob:L_q} is $x_{i^*}$ with $i^* = \argmin_{0\leq j\leq k}f(x_j)$.

\begin{remark}\label{rm:Lq}
	\label{L_q_Modify}
The $\ell_q$ regularization is not locally Lipschitz at 0 when $0<q<1$, which might cause problems. However, if we replace $\|x\|^q$ with $\min\{|x|^q,B|x|\},B\gg0$, then the new regularization is locally Lipschitz on $\mathbb{R}$, and it differs from the original function only on  $(-\frac{1}{B^{1-q}},+\frac{1}{B^{1-q}})$. The closed-form solution can still be obtained by comparing the objective values at $x_1^* = \argmin_{x} ax^2+bx+c|x|^q$ and $x_2^* = \argmin_{x} ax^2+bx+cB|x| = \left(\frac{-cB-b}{2a}\right)_+$.
Actually due to the limited machine precision, the window $(-\frac{1}{B^{1-q}},+\frac{1}{B^{1-q}})$ shrinks to a single point $0$ when $B$ is sufficiently large. Since this causes no numerical difficulties, we can just deal with $\ell_q$ penalties by replacing it by the modified version.
\end{remark}

\subsection{The community detection problem}
Given any undirected network, the community detection problem aims to figure out the clusters, in other words the communities, of this network; see for example \cite{CMM,SCORE,OCCAM,CP-comm}, etc. A viable way to solve this problem is via the symmetric othorgonal nonnegative matrix approximation. Suppose the adjacency matrix of the network is $A$ , then the method aims to solve
\bea
\label{prob:community}
\min_{X\in\mathbb{R}^{n\times k}}\|A-XX^\top\|^2_F, \mbox{  \st  } X^\top X = I_{k\times k}, \mbox{  }X\geq 0,
\eea
where $n$ equals the number of nodes and $k$ equals the number of communities. When the network is connected, the orthogonality and nonnegativeness of the optimal solution $X^*$ indicate that there is exactly one positive entry in each row of $X^*$. Therefore we can reconstruct the community structure by letting node $i$ belong to community $j$ if $X^*_{ij}>0$.

In our framework, this problem can be naturally formulated as
\bea
\label{prob:com-ADMM}
& \min_{X,Y,Z\in\mathbb{R}^{n\times k}} & \|A - XX^\top\|_F^2 + \frac{\mu}{2}\|Z\|_F^2 \nonumber\\
& \st & X^\top X = I_{k\times k}, \mbox{  } Y \geq 0, \\
& & X-Y + Z = 0,\nonumber
\eea
where the orthogonal $X$ is forced to be equal to the nonnegative $Y$, while a slack variable $Z$ is added so that they do not need to be exactly equal. In the implementation of the Algorithm \ref{alg:PADMM-L}, two subproblems for block $X$ and $Y$ need to be solved. For the orthogonal block $X$, the subproblem is still in the form of \eqref{subprob:stiefel}. For the nonnegative block $Y$, the subproblem can be formulated as:
\be
Y^* = \arg\min_{Y\geq0} \|Y-B\|_F^2 = B_+,
\ee
for some matrix $B$. The notation $B_+$ is defined by $B_+ = \max\{B,0\}$, where the $\max$ is taken elementwise.

\section{Numerical Results}\label{sec:Num_rst}

\subsection{The maximum bisection problem}
We consider the following variant of maximum bisection problem to apply our proposed algorithm.
\[
\begin{array}{lll}
& \min_{U,z,x} & \langle W,UU^\top \rangle + \frac{\mu}{2}\|z\|^2 \nonumber \\
& \text{s.t.} & \|u_i\|^2 = 1, u_i\geq 0, \mbox{ for } i = 1,...,n,  \nonumber\\
& & \sum_{i=1}^nu_i - x\mathbf{1} + z = 0, \\
& & z\in\R^2 \mbox{ is free, } \frac{n}{2} - \nu \leq x \leq  \frac{n}{2} + \nu, \nonumber
\end{array}
\]
where $\nu\geq 0$ is a parameter that controls the tightness of the relaxation.
In our experiments, we set $\nu = 1$. We choose five graphs from the maximum cut library \emph{Biq Mac Library} \cite{biqmac-library} to test our algorithm, with the following specifics in Table \ref{tab:graph-Info}.

\begin{table}[!htbp]
	\centering
	\begin{tabular}{|c|c|c|c|c|c|}
		\hline
		\multicolumn{6}{|c|}{Graph Information}\\
		\hline
		Network & g05\_60.0 & g05\_80.0 & g05\_100.0 & pw01\_100.0 & pw09\_100.0\\
		\hline
		\# nodes& 60 & 80 & 100 & 100 & 100 \\
		\hline
		\# edges& 885 & 1580 & 2475 & 495 & 4455\\
		\hline
	\end{tabular}
	\caption{The test graph information.}\label{tab:graph-Info}
\end{table}

For the three tested algorithms, we denote the SDP relaxation proposed by Frieze \etal in \cite{Max-bisec:Frieze-1997} as SDP-F, we denote the SDP relaxation proposed by Ye in \cite{Max-bisec:Ye-2011} as SDP-Y, and we denote our low-rank relaxation as LR. The SDP relaxations are solved by the interior point method embedded in CVX \cite{cvx}. To solve the problem by our proposed Algorithm \ref{alg:PADMM}, we set $\mu = 0.01.$ Other parameters such as $\beta,\gamma,H_i = \sigma I$ are chosen according to our theories for given estimation of the Lipschitz constant $L$. For all cases, the number of iterations is set to 30. For each graph, all algorithms are tested for 20 times and then we compare their average cut values. The results are reported in Table \ref{tab:max-bisection}.

\begin{table}[!htbp]
	\centering
	\begin{tabular}{|c|c|c|c|c|c|c|}
		\hline
		Network & avg LR cut & SD & avg SDP-Y cut & $\text{ratio}_1$ & avg SDP-F cut & $\text{ratio}_2$\\
		\hline
		g05\_60.0& 1051.3 & 15.9773 & 1033.2 & 1.0175 & 1045.4 & 1.0056 \\
		\hline
		g05\_80.0& 1822.7 & 15.3180 & 1778.5 & 1.0249 & 1805.9 & 1.0093\\
		\hline
		g05\_100.0& 2810.2 & 19.4413 & 2775.7 & 1.0124 & 2799.8 & 1.0037\\
		\hline
		pw01\_100.0& 3946.8 & 28.5032 & 3889.7 & 1.0147 & 3944.3 & 1.0006\\
		\hline
		pw09\_100.0& 26863.2 & 102.1318 & 26609 &  1.0096 & 26764.1 & 1.0037\\
		\hline
	\end{tabular}
	\caption{The column \emph{SD} contains the standard deviations of the LR cut values in 20 rounds. $\text{ratio}_1 =\frac{\text{avg LR cut}}{\text{avg SDP-Y cut}}$, and $\text{ratio}_2=\frac{\text{avg LR cut}}{\text{avg SDP-F cut}}$.}\label{tab:max-bisection}
\end{table}
It is interesting to see that in all tested cases, our proposed relaxation solved by Algorithm \ref{alg:PADMM} outperforms the two SDP relaxations in \cite{Max-bisec:Frieze-1997,Max-bisec:Ye-2011}. Moreover, our method is a first-order method, and it naturally enjoys computational advantages compared to the interior-point based methods for solving the SDP relaxation.

Finally, in this application we test the performance of Algorithm \ref{alg:PADMM-L} by comparing it to Algorithm \ref{alg:PADMM}. We keep the parameters $\mu,\beta,\gamma,\nu$ 
unchanged
for testing Algorithm \ref{alg:PADMM-L}, but we reset $H_i = \sigma I$ according to its new bound in Theorem \ref{thm:PADMM-L}. For each graph, 20 instances are tested, and 30 iterations are performed for each algorithm. The objective measured is $\langle W,UU^\top \rangle$. The result is shown in Table \ref{table:admm-l-maxbis}. It can be observed that in this case, Algorithm \ref{alg:PADMM-L} behaves similarly as Algorithm~\ref{alg:PADMM}.

\begin{table}[h]
	\centering
	\begin{tabular}{|*{5}{c|}}
		\hline
		\multicolumn{1}{|c|}{\multirow{2}*{Network}}
		& \multicolumn{2}{|c|}{Algorithm \ref{alg:PADMM}} & \multicolumn{2}{|c|}{Algorithm \ref{alg:PADMM-L}}
		\\\cline{2-5}
		\multicolumn{1}{|c|}{}  &avg obj & SD &avg obj& SD  \\\hline
		g05\_60.0	 &  724.2    &   13.4070   &  719.7   &  12.3164\\\hline
		g05\_80.0	 &  1335    &   9.6791   &  1340.7   &  18.8766\\\hline
		g05\_100.0	 &  2136.1    &   24.6446   &  2135.5   &  18.8275\\\hline
		pw01\_100.0	 &  1558.8    &   78.0591   &  1563.7   &  76.5748\\\hline
		pw09\_100.0	 &  22262.8    &   100.1208   &  22371.3   &  119.8688\\\hline
	\end{tabular}
	\caption{Numerical performance of Algorithm \ref{alg:PADMM-L} for problem \eqref{prob:max-bisec-0}.}\label{table:admm-l-maxbis}
\end{table}

\subsection{The $\ell_q$ regularized sparse tensor PCA}
In this experiment, we synthesize a set of ground truth Tucker format tensors $\bT^{(i)}_{true} =  \bC^{(i)}\times_{1}U_1\times_2\cdots\times_{d}U_d$, where all $\bT^{(i)}_{true}$'s share the same factors $U_j$ while having different cores $\bC^{(i)}$. We test our methods by two cases, the first set of tensors have mode sizes $30\times30\times30$ and core mode sizes $5\times5\times5$. The second set of tensors have  mode sizes $42\times42\times42$ and core mode sizes $7\times7\times7$. For both cases, we generate 100 instances. We associate a componentwise Gaussian white noise $\bT^{(i)}_{noise}$ with standard deviation $0.001$ to each tensor. Namely, the input data are $\bT^{(i)} = \bT^{(i)}_{true} + \bT_{noise}^{(i)}, ~i = 1,...,100.$ For all cases, the core elements are generated by uniform distribution in $[-1,1]$. The sparsity level of each core $\bC^{(i)}$ is set to $0.3$, i.e., we randomly set 70\% of the elements to zero in each core. Finally, the orthogonal factors $U_i$ are generated with sparsity level $1/6$.

To solve \eqref{prob:MPCA}, we set the regularization terms to $\ell_{2/3}$ penalties for cores and to $\ell_{1}$ penalties for the factors. That is, $q=2/3$ and $p=1$ in \eqref{prob:MPCA}. The sparse penalty parameters are set to $\alpha_1 = 0.1$ and $\alpha_2 = 0.01$. We set $\mu = 10^{-6}$, and other parameters $\beta,\gamma,H_i = \sigma I$ are chosen according to our theories for given estimation of the Lipschitz constant $L$.

Our numerical results show that it is indeed necessary to set different regularizations for cores and factors. In the output of the result, the matrices $U_i$'s are definitely not sparse, but with plenty of entries very close to 0. The output $V_i$'s are very sparse but are not orthogonal. We construct the final output from $U_i$ by zeroing out all the entries with absolute value less than 0.001. Then the resulting matrices $\bar{U}_i$'s are sparse and are almost orthogonal. Finally, the relative error is measured using $\bar{U}_i$ and the underlying true tensor, i.e., $\frac{1}{100}\sum_{i=1}^{100}\frac{\|\bT^{(i)}_{true} - \bT^{(i)}_{out}\|^2}{\|\bT^{(i)}_{true}\|^2}$, where $\bT^{(i)}_{out}$'s are constructed from the output of the algorithms. The orthogonality violation is measured by $\frac{1}{3}\sum_{i=1}^{3}\|\bar{U}_i^\T \bar{U}_i-I\|_F$. In both cases, the iteration number is set to be 100. For each case, 10 instances are generated and we report the average performance in Table \ref{table:lq-MPCA}. The results are obtained from 20 randomly generated instances. The columns $err_1$, $SD$, $err_2$, $spars_1$, $spars_2$ denote the averaged objective relative errors, the standard deviation of the objective relative errors, the average orthogonality constraint violation, the average core sparse levels and the average factor sparse levels respectively. 

\begin{table}[h]
	\centering
	\begin{tabular}{|*{10}{c|}}
		\hline
		\multicolumn{5}{|c|}{$30\times30\times30 $, core $5\times5\times5$} & \multicolumn{5}{|c|}{$42\times42\times42 $, core $7\times7\times7$}
		\\\hline
		$avg~err_1$ & SD &$err_2$& $spars_1$ &$spars_2$&$err_1$ & SD &$err_2$& $spars_1$ &$spars_2$\\\hline
		0.0043 & 0.0028	 &  $2.7\times10^{-7}$ &  0.5363 &  1/6  &  0.0803 & 0.0010 & $1.2\times10^{-14}$  &  0.5387  & 1/6 \\\hline
	\end{tabular}
	\caption{Numerical performance of Algorithm \ref{alg:PADMM} for problem \eqref{prob:MPCA}. }
	\label{table:lq-MPCA}
\end{table}

\subsection{The community detection problem}
For this problem, we test our algorithm on three real world  social networks with ground truth information. They are the American political blogs network with 1222 nodes and 2 communities specified by their political leaning, the Caltech facebook network with 597 nodes and 8 communities specified by their dorm number, and the Simmons College facebook network with 1168 nodes and 4 communities specified by their graduation years. Note that \eqref{prob:com-ADMM} is a very simple model, so we will not compare it the more sophisticated models such as \cite{CMM,CP-comm}. Instead it is compared with the state-of-the-art spectral methods SCORE \cite{SCORE} and OCCAM \cite{OCCAM}.

In all tests for the three networks, the parameter $\mu$ is set to be 50 and $L$ is set to be $100$. The other parameters $\beta,\gamma,H_i = \sigma I$ are chosen according to our theories for a given estimation of $L$. For each network, every algorithm is run for 20 times and the average error rate is reported in Table \ref{tab:comm}.

\begin{table}[!htbp]
	\centering
	\begin{tabular}{|c||c|c|c|}
		\hline
		Network Name & Algorithm \ref{alg:PADMM-L} & SCORE & OCCAM \\
		\hline
		Polblogs& 5.07\% & \textbf{4.75\%} & 4.91\% \\
		\hline
		Caltech& \textbf{23.68\%} & 28.66\% & 34.21\%\\
		\hline
		Simmons& \textbf{20.61\%} & 22.54\% & 23.92\% \\
		\hline
	\end{tabular}
	\caption{Numerical performance of Algorithm \ref{alg:PADMM-L} for problem \eqref{prob:com-ADMM}.}\label{tab:comm}
\end{table}
It can be observed from the numerical results that Algorithm \ref{alg:PADMM-L} yields the best result in Caltech and Simmons College networks, and is only slightly outperformed in the political blogs network, which shows the effectiveness of our method for this problem.

\section{Conclusions}
In this paper we extend the framework studied in \cite{NcvxADMM:Zhang-etal-2016} and develop a proximal ADMM-like algorithm for nonsmooth and nonconvex multi-block optimization over Riemannian manifolds. It turns out that this model has a wide range of applications. The linearized and the stochastic as well as the curvilinear line-search-based variants of this algorithm are proposed to handle the situations where exact minimization is hard, or the function/gradient evaluation is expensive. For all the proposed algorithms, an $\mathcal{O}(1/\epsilon^2)$ iteration complexity is guaranteed. The numerical experiments show great potential of the proposed methods. It is worth noting that when the problem is not in the form of \eqref{prob:main},
then the reformulation proposed in Section \ref{sec:Extension} will in general lead to an increased iteration complexity.

\bibliography{Manifold-10}
\bibliographystyle{plain}

\appendix
\section{Proofs of the technical lemmas}

\subsection{Proof of Lemma \ref{lm:PADMM-lemma2}}

\begin{proof}
By the global optimality for the subproblems in Step 1 of Algorithm \ref{alg:PADMM}, we have
\be
\label{lm_:PADMM-lemma2:1}
\cL_{\beta}(\xke_1,\cdots,\xke_{N-1},\xk_N,\lambda^k)\leq \cL_{\beta}(\xk_1,\cdots,\xk_{N-1},\xk_N,\lambda^k) - \frac{1}{2}\sum_{i=1}^{N-1}\|\xk_i-\xke_i\|^2_{H_i}.
\ee

By Step 2 of Algorithm \ref{alg:PADMM} we have
\be
\label{lm_:PADMM-lemma2:2}
\cL_{\beta}(\xke_1,\cdots,\xke_{N-1},\xke_N,\lambda^k)\leq \cL_{\beta}(\xke_1,\cdots,\xke_{N-1},\xk_N,\lambda^k) + \left(\frac{L+\beta}{2}-\frac{1}{\gamma}\right)\|\xk_N-\xke_N\|^2.
\ee
By Step 3, directly substitute $\lambda^{k+1}$ into the augmented Lagrangian gives
\be
\label{lm_:PADMM-lemma2:3}
\cL_{\beta}(\xke_1,\cdots,\xke_N,\lambda^{k+1})= \cL_{\beta}(\xke_1,\cdots,\xke_N,\lambda^k) + \frac{1}{\beta}\|\lambda^k-\lambda^{k+1}\|^2.
\ee

Summing up \eqref{lm_:PADMM-lemma2:1}, \eqref{lm_:PADMM-lemma2:2}, \eqref{lm_:PADMM-lemma2:3}) and apply Lemma \ref{lm:PADMM-lemma1},
we obtain the following inequality,
\bea
\label{lm_:PADMM-lemma2:4}
& & \cL_{\beta}(\xke_1,\cdots,\xke_{N-1},\xke_N,\lambda^{k+1})- \cL_{\beta}(\xk_1,\cdots,\xk_{N-1},\xk_N,\lambda^k) \nonumber\\
& \leq & \left[\frac{L+\beta}{2}-\frac{1}{\gamma}+\frac{3}{\beta}\left(\beta-\frac{1}{\gamma}\right)^2\right]\|\xk_N-\xke_N\|^2   \\
& & +\frac{3}{\beta}\left[\left(\beta-\frac{1}{\gamma}\right)^2+L^2\right]\|\xkm_N-\xk_N\|^2 - \sum_{i = 1}^{N-1}\|\xk_i-\xke_i\|^2_{\frac{1}{2}H_i-\frac{3L^2}{\beta}I},\nonumber
\eea
which further indicates
\bea
& & \Psi_G(\xke_1,\cdots,\xke_{N-1},\xke_N,\lambda^{k+1},\xk_N) - \Psi_G(\xk_1,\cdots,\xk_{N-1},\xk_N,\lambda^{k},\xkm_N) \nonumber\\
&\leq& \left[\frac{\beta+L}{2}-\frac{1}{\gamma}+\frac{6}{\beta}\left(\beta-\frac{1}{\gamma}\right)^2+\frac{3L^2}{\beta}\right]\|\xk_N-\xke_N\|^2 \\
& & - \sum_{i = 1}^{N-1}\|\xk_i-\xke_i\|^2_{\frac{1}{2}H_i-\frac{3L^2}{\beta}I}. \nonumber
\eea

To ensure that the right hand side of \eqref{lm_:PADMM-lemma2:5} is negative, we need to choose $H_i\succ\frac{6L^2}{\beta}I$, and ensure that
\be\label{lm_:PADMM-lemma2:6}
\frac{\beta+L}{2}-\frac{1}{\gamma}+\frac{6}{\beta}\left(\beta-\frac{1}{\gamma}\right)^2+\frac{3L^2}{\beta}<0.
\ee
This can be proved by first viewing it as a quadratic function of $z = \frac{1}{\gamma}$. To find some $z>0$ such that
$$p(z) = \frac{6}{\beta}z^2 - 13z + \left(\frac{L+\beta}{2}+6\beta+\frac{3}{\beta}L^2\right)<0,$$
we need the discriminant to be positive, i.e.
\be
\label{delta}\Delta(\beta) = \frac{1}{\beta^2}(13\beta^2-12\beta L -72L^2)>0.
\ee
It is easy to verify that \eqref{beta}
suffices to guarantee \eqref{delta}.
Solving $p(z) = 0$, we find two positive roots
\[
z_{1} = \frac{13\beta-\sqrt{13\beta^2-12\beta L -72L^2}}{12}, \mbox{ and } z_{2} = \frac{13\beta+\sqrt{13\beta^2-12\beta L -72L^2}}{12}.
\]
Note that $\gamma$ defined in \eqref{gamma} satisfies $\frac{1}{z_2}<\gamma<\frac{1}{z_1}$ and thus guarantees \eqref{lm_:PADMM-lemma2:6}. This completes the proof.
\end{proof}

\subsection{Proof of Lemma \ref{lm:PADMM-L-lemma2}}

\begin{proof}
For the subproblem in Step 1 of Algorithm \ref{alg:PADMM-L}, since $x_i^{k+1}$ is the global minimizer, we have
\begin{eqnarray*}
	& & \langle\nabla_if(\xke_1,\cdots,\xke_{i-1},\xk_i,\cdots,\xk_N),\xke_i-\xk_i\rangle -\bigg\langle\sum_{j=1}^{i}A_j\xke_j+\sum_{j = i+1}^{N}A_j\xk_j-b,\lambda^k \bigg\rangle\\
	& & +\frac{\beta}{2}\bigg\|\sum_{j=1}^{i}A_j\xke_j+\sum_{j = i+1}^{N}A_j\xk_j-b\bigg\|^2  + \sum_{j=1}^{i}r_j(\xke_j)+\sum_{j = i+1}^{N-1}r_j(\xk_j)\\
	& \leq &
	-\bigg\langle\sum_{j=1}^{i-1}A_j\xke_j+\sum_{j = i}^{N}A_j\xk_j-b,\lambda^k \bigg\rangle +\frac{\beta}{2}\bigg\|\sum_{j=1}^{i-1}A_j\xke_j+\sum_{j = i}^{N}A_j\xk_j-b\bigg\|^2 \\
	& & + \sum_{j=1}^{i-1}r_j(\xke_j)+\sum_{j = i}^{N-1}r_j(\xk_j)-\frac{1}{2}\|\xke_i-\xk_i\|^2_{H_i}.
\end{eqnarray*}

By the $L$-Lipschitz continuity of $\nabla_i f$, we have 
\begin{eqnarray*}
	& & f(\xke_1,\cdots,\xke_{i},\xk_{i+1},\cdots,\xk_N)\\ &\leq& f(\xke_1,\cdots,\xke_{i-1},\xk_i,\cdots,\xk_N)
	+\langle\nabla_if(\xke_1,\cdots,\xke_{i-1},\xk_i,
\cdots,\xk_N),\xke_i-\xk_i\rangle \\
	& & +\frac{L}{2}\|\xke_i-\xk_i\|^2.
\end{eqnarray*}

Combining the above two inequalities and using the definition of $\cL_{\beta}$ in \eqref{Lagrangian}, we have
\be
\label{lm_:PADMM-L-lemma2:additional}
\cL_{\beta}(\xke_1,\cdots,\xke_{i},\xk_{i+1},\cdots,\xk_N,\lambda^k)\leq\cL_{\beta}(\xke_1,\cdots,\xke_{i-1},\xk_i,\cdots,
\xk_N,\lambda^k)-\|\xk_i-\xke_i\|^2_{\frac{H_i}{2}-\frac{L}{2}I}.
\ee
Summing \eqref{lm_:PADMM-L-lemma2:additional} over $i=1,\ldots,N-1$, we have the following inequality, which is the counterpart of \eqref{lm_:PADMM-lemma2:1}:
\be
\label{lm_:PADMM-L-lemma2:1}
\cL_{\beta}(\xke_1,\cdots,\xke_{N-1},\xk_N,\lambda^k)\leq\cL_{\beta}(\xk_1,\cdots,\xk_N,\lambda^k)-\sum_{i=1}^{N-1}\|\xk_i-\xke_i\|^2_{\frac{H_i}{2}-\frac{L}{2}I}.
\ee

Besides, since \eqref{lm_:PADMM-lemma2:2} and \eqref{lm_:PADMM-lemma2:3} still hold, by combining \eqref{lm_:PADMM-L-lemma2:1}, \eqref{lm_:PADMM-lemma2:2} and \eqref{lm_:PADMM-lemma2:3} and applying Lemma \ref{lm:PADMM-lemma1}, we establish the following two inequalities, which are respectively the counterparts of \eqref{lm_:PADMM-lemma2:4} and \eqref{lm_:PADMM-lemma2:5}:
\bea
\label{lm_:PADMM-L-lemma2:2}
& & \cL_{\beta}(\xke_1,\cdots,\xke_{N-1},\xke_N,\lambda^{k+1})- \cL_{\beta}(\xk_1,\cdots,\xk_{N-1},\xk_N,\lambda^k) \nonumber\\
& \leq & \left[\frac{L+\beta}{2}-\frac{1}{\gamma}+\frac{3}{\beta}\left(\beta-\frac{1}{\gamma}\right)^2\right]\|\xk_N-\xke_N\|^2   \\
& & +\frac{3}{\beta}\left[\left(\beta-\frac{1}{\gamma}\right)^2+L^2\right]\|\xkm_N-\xk_N\|^2 - \sum_{i = 1}^{N-1}\|\xk_i-\xke_i\|^2_{\frac{1}{2}H_i-\frac{L}{2}I-\frac{3L^2}{\beta}I},\nonumber
\eea
and
\bea
& & \Psi_G(\xke_1,\cdots,\xke_{N-1},\xke_N,\lambda^{k+1},\xk_N) - \Psi_G(\xk_1,\cdots,\xk_{N-1},\xk_N,\lambda^{k},\xkm_N) \nonumber\\
&\leq& \left[\frac{\beta+L}{2}-\frac{1}{\gamma}+\frac{6}{\beta}\left(\beta-\frac{1}{\gamma}\right)^2+\frac{3L^2}{\beta}\right] \|\xk_N-\xke_N\|^2 \nonumber\\
& & - \sum_{i = 1}^{N-1}\|\xk_i-\xke_i\|^2_{\frac{1}{2}H_i-\frac{L}{2}I-\frac{3L^2}{\beta}I}.\nonumber
\eea
From the proof of Lemma \ref{lm:PADMM-lemma2}, it is easy to see that the right hand side of the above inequality is negative, if
$H_i\succ \left(\frac{6L^2}{\beta}+L\right)I$
and $\beta$ and $\gamma$ are chosen according to \eqref{beta} and \eqref{gamma}.
\end{proof}

\subsection{Proof of Lemma \ref{lm:PADMM-S-lemma1}}

\begin{proof}
For the ease of notation, we denote
\be
\label{error}
G_i^M(x_1^{k+1},\ldots,x_{i-1}^{k+1},x_i^k,\ldots,x_N^k) = \nabla_i f(x_1^{k+1},\ldots,x_{i-1}^{k+1},x_i^k,\ldots,x_N^k)+\delta_i^k.
\ee
Note that $\delta_i^k$ is a zero-mean random variable. By Steps 2 and 3 of Algorithm \ref{alg:PADMM-S} we obtain
\be
\label{sto:1-1}
\lambda^{k+1} = \left(\beta-\frac{1}{\gamma}\right)(x_N^k-x_N^{k+1})+\nabla_Nf(\xke_1,\cdots,\xke_{N-1},\xk_N)+\delta_N^k.
\ee
Applying \eqref{sto:1-1} for $k$ and $k+1$, and using \eqref{sto:1-1}, we get
\bea
\|\lambda^{k+1}-\lambda^k\|^2 & = & \bigg\|\left(\beta-\frac{1}{\gamma}\right)(\xk_N-\xke_N)-\left(\beta-\frac{1}{\gamma}\right)(\xkm_N-\xk_N)+(\delta_N^k-\delta_N^{k-1})\nonumber\\
& & +(\nabla_Nf(\xke_1,\cdots,\xke_{N-1},\xk_N)-\nabla_Nf(\xk_1,\cdots,\xk_{N-1},\xkm_N)\bigg\|^2 \nonumber\\
& \leq & 4\left(\beta-\frac{1}{\gamma}\right)^2\|\xk_N-\xke_N\|^2+4\left[\left(\beta-\frac{1}{\gamma}\right)^2+L^2\right]\|\xkm_N-\xk_N\|^2 \nonumber \\
& & +4L^2\sum_{i=1}^{N-1}\|\xk_i-\xke_i\|^2 + 4\|\delta_N^k-\delta_N^{k-1}\|^2. \nonumber
\eea
Taking expectation with respect to all random variables
on both sides and using $\E[\langle\delta_N^k,\delta_N^{k-1}\rangle] = 0$ completes the proof.
\end{proof}

\subsection{Proof of Lemma \ref{lm:PADMM-S-lemma2}}

\begin{proof}
Similar as \eqref{lm_:PADMM-L-lemma2:additional}, by further incorporating \eqref{error}, we have
\begin{eqnarray*}
	& & \cL_{\beta}(\xke_1,\cdots,\xke_{i},\xk_{i+1},\cdots,\xk_N,\lambda^k)-\cL_{\beta}(\xke_1,\cdots,\xke_{i-1},\xk_i,\cdots,\xk_N,\lambda^k)\\
	& \leq & -\|\xk_i-\xke_i\|^2_{\frac{H_i}{2}-\frac{L}{2}I}+\langle\delta_i^k,\xke_i-\xk_i\rangle \\
	& \leq & -\|\xk_i-\xke_i\|^2_{\frac{H_i}{2}-\frac{L}{2}I}+\frac{1}{2}\|\delta_i^k\|^2+\frac{1}{2}\|\xke_i-\xk_i\|^2.
\end{eqnarray*}
Taking expectation  with respect to all random variables on both sides and summing over $i=1,\ldots,N-1$, and using \eqref{BSFO}, we obtain
\begin{eqnarray}
\label{sto:decrease1}
& & \E[\cL_{\beta}(\xke_1,\cdots,\xke_{N-1},\xk_N,\lambda^k)]-\E[\cL_{\beta}(\xk_1,\cdots,\xk_N,\lambda^k)] \\
& \leq & -\sum_{i=1}^{N-1}\E\left[\|\xke_i-\xk_i\|^2_{\frac{1}{2}H_i-\frac{L+1}{2}I}\right]+\frac{N-1}{2M}\sigma^2.\nonumber
\end{eqnarray}
Note that by the Step 2 of Algorithm \ref{alg:PADMM-S} and the descent lemma we have
\bea
0 & = & \bigg\langle x_N^k - x_N^{k+1}, \nabla_N f(x_1^{k+1},...,x_{N-1}^{k+1},x_N^k) + \delta_N^k - \lambda^k + \beta\left(\sum_{j=1}^{N-1}A_jx_j^{k+1}+x_N^k-b\right) - \frac{1}{\gamma}(x_N^k-x_N^{k+1})\bigg\rangle \nonumber\\
& \leq  & f(x_1^{k+1},...,x_{N-1}^{k+1},x_N^k) - f(x^{k+1}) + \left(\frac{L+\beta}{2}-\frac{1}{\gamma}\right)\|x_N^{k+1}-x_N^k\|^2 - \langle\lambda^k,x_N^k-x_N^{k+1}\rangle \nonumber\\
& & + \frac{\beta}{2}\|\sum_{j=1}^{N-1}A_jx_j^{k+1}+x_N^k-b\|^2 - \frac{\beta}{2}\|\sum_{j=1}^{N-1}A_jx_j^{k+1}+x_N^{k+1}-b\|^2 + \langle\delta_N^k,x_N^k-x_N^{k+1}\rangle \nonumber \\
& \leq &  \cL_{\beta}(\xke_1,\cdots,\xke_{N-1},\xk_N,\lambda^k) -  \cL_{\beta}(\xke,\lambda^k) + \left(\frac{L+\beta}{2}-\frac{1}{\gamma}+\half\right)\|\xk_N-\xke_N\|^2 + \half\|\delta_N^k\|^2. \nonumber
\eea
Taking the expectation with respect to all random variables yields
\bea
\label{sto:decrease2}
& & \E[\cL_{\beta}(\xke_1,\cdots,\xke_{N-1},\xke_N,\lambda^k)]- \E[\cL_{\beta}(\xke_1,\cdots,\xke_{N-1},\xk_N,\lambda^k)]\nonumber\\
&\leq&  \left(\frac{L+\beta}{2}-\frac{1}{\gamma}+\half\right)\E[\|\xk_N-\xke_N\|^2]+\frac{1}{2M}\sigma^2.
\eea
The following equality holds trivially from Step 3 of Algorithm \ref{alg:PADMM-S}:
\be
\label{sto:decrease3}
\E[\cL_{\beta}(\xke_1,\cdots,\xke_N,\lambda^{k+1})]-\E[\cL_{\beta}(\xke_1,\cdots,\xke_N,\lambda^{k})] = \frac{1}{\beta}\E[\|\lambda^k-\lambda^{k+1}\|^2].
\ee
Combining \eqref{sto:decrease1}, \eqref{sto:decrease2}, \eqref{sto:decrease3} and \eqref{sto:1}, we obtain
\bea
& & \E[\Psi_S(\xke_1,\cdots,\xke_{N-1},\xke_N,\lambda^{k+1},\xk_N)] - \E[\Psi_S(\xk_1,\cdots,\xk_{N-1},\xk_N,\lambda^{k},\xkm_N)] \label{Psi-S-decrease}\\
&\leq & \left[\frac{\beta+L}{2}-\frac{1}{\gamma}+\frac{8}{\beta}\left(\beta-\frac{1}{\gamma}\right)^2+\frac{4L^2}{\beta}+\half\right] \E [\|\xk_N-\xke_N\|^2] \nonumber\\
& & - \sum_{i = 1}^{N-1}\E\left[\|\xk_i-\xke_i\|^2_{\frac{1}{2}H_i-\frac{4L^2}{\beta}I-\frac{L+1}{2}I}\right]+\left(\frac{8}{\beta}+\half+\frac{N-1}{2}\right)\frac{\sigma^2}{M}.\nonumber
\eea
Choosing $\beta$ and $\gamma$ according to \eqref{sto:beta} and \eqref{sto:gamma}, and using the similar arguments in the proof of Lemma \ref{lm:PADMM-lemma2}, it is easy to verify that
$$\left[\frac{\beta+L}{2}-\frac{1}{\gamma}+\frac{8}{\beta}\left(\beta-\frac{1}{\gamma}\right)^2+\frac{4L^2}{\beta}+\half\right]<0.$$
By further choosing $H_i\succ\left(\frac{8L^2}{\beta}+L+1\right)I$, we know that the right hand side of \eqref{Psi-S-decrease} is negative, and this completes the proof.
\end{proof}

\subsection{Proof of Lemma \ref{lm:PADMM-S-lemma3}}

\begin{proof}
From \eqref{sto:1-1} and \eqref{eq:assumption-2-lips}, we have that
\beaa
& & \cL_{\beta}(\xke_1,\cdots,\xke_N,\lambda^{k+1})\nonumber\\
& = & \sum_{i = 1}^{N-1}r_i(\xke_i) + f(\xke) - \bigg\langle \sum_{i = 1}^NA_i\xke_i-b, \nabla_Nf(\xke) +  \left(\beta-\frac{1}{\gamma}\right)(\xk_N-\xke_N) \nonumber\\
& & + \nabla_Nf(\xke_1,...,\xke_{N-1},\xk_N) - \nabla_Nf(\xke)+\delta_N^k\bigg\rangle + \frac{\beta}{2}\bigg\|\sum_{i = 1}^NA_i\xke_i-b\bigg\|^2  \nonumber\\
& \geq & \sum_{i = 1}^{N-1}r_i(\xke_i) + f(\xke_1,...,\xke_{N-1}, b-\sum_{i=1}^{N-1}A_i\xke_i) -\frac{4}{\beta}\left[\left(\beta-\frac{1}{\gamma}\right)^2+L^2\right]\|\xk_N-\xke_N\|^2          \nonumber  \\
& &  + \bigg(\frac{\beta}{2}-\frac{\beta}{8}-\frac{\beta}{8}-\frac{L}{2}\bigg)\bigg\|\sum_{i = 1}^NA_i\xke_i-b\bigg\|^2  - \frac{2}{\beta}\|\delta_N^k\|^2    \nonumber      \\
& \geq & \sum_{i=1}^{N-1}r_i^*+f^*-\frac{4}{\beta}\left[\left(\beta-\frac{1}{\gamma}\right)^2+L^2\right]\|\xk_N-\xke_N\|^2 - \frac{2}{\beta}\|\delta_N^k\|^2 \nonumber\\
\eeaa
where the first inequality is obtained by applying
$\langle a, b\rangle\leq\half(\frac{1}{\eta}\|a\|^2+\eta \|b\|^2)$ to terms $\langle \sum_{i = 1}^NA_i\xke_i-b,  \left(\beta-\frac{1}{\gamma}\right)(\xk_N-\xke_N)\rangle$, $\langle \sum_{i = 1}^NA_i\xke_i-b, \nabla_Nf(\xke_1,...,\xke_{N-1},\xk_N) - \nabla_Nf(\xke)\rangle$ and $\langle \sum_{i = 1}^NA_i\xke_i-b,\delta_N^k\rangle$ respectively with $\eta = \frac{8}{\beta}, \frac{8}{\beta}$ and $\frac{4}{\beta}$. Note that $\beta>2L$ according to \eqref{sto:beta}, thus $(\frac{\beta}{2}-\frac{\beta}{8}-\frac{\beta}{8}-\frac{L}{2})>0$ and the last inequality holds. By rearranging the terms and taking expectation with respect to all random variables completes the proof.
\end{proof}

\subsection{Proof for Theorem \ref{thm:retraction}}
\begin{proof}
	Through similar argument, one can easily obtain
	$$\|\lambda^{k+1} - \nabla_N f(\xke_1,\cdots,\xke_N)\|^2\leq \kappa_2\theta_k\quad \mbox{ and }\quad \left\|\sum_{i=1}^{N-1}A_i\xke_i+\xke_N-b\right\|^2 \leq \kappa_1\theta_k.$$
	The only remaining task is to guarantee an $\epsilon$ version of \eqref{opt_ADMM_retraction}. First let us prove that
	\be
	\label{gradient-bound}
	\|g_i^{k+1}\| \leq \frac{\sigma+2L_2C+(L+\beta A_{\max}^2)L_1^2}{2\alpha}\sqrt{\theta_{k}}.
	\ee
	Denote $h_i(x_i) = \cL_{\beta}(x^{k+2}_1,...,x^{k+2}_{i-1},x_i,x^{k+1}_{i+1},...,x^{k+1}_N,\lambda^{k+1})$ and $Y_i(t) = R(x^{k+1}_i,-tg_i^{k+1})$, then it is not hard to see that $\nabla h_i(x_i)$ is Lipschitz continuous with parameter $L+\beta \|A_i\|_2^2 \leq L_3:=L+\beta A_{\max}^2$. Consequently, it yields
	\beaa
	h_i(Y_i(t)) & \leq & h_i(Y_i(0)) + \langle \nabla h_i(Y_i(0)), Y_i(t) - Y_i(0) - tY_i'(0) + tY'_i(0)\rangle + \frac{L_3}{2} \|Y_i(t) - Y_i(0)\|^2 \\
	& \leq & h_i(Y_i(0)) + t\langle\nabla h_i(Y_i(0)),Y_i'(0)\rangle + L_2t^2\|\nabla h_i(Y_i(0))\|\|Y'_i(0)\|^2 + \frac{L_3L_1^2}{2}t^2\|Y'_i(0)\|^2 \\
	& = & h_i(Y_i(0)) - \left(t-L_2t^2\|\nabla h_i(Y_i(0))\| - \frac{L_3L_1^2}{2}t^2\right)\|Y'_i(0)\|^2,
	\eeaa
	where the last equality is due to $\langle\nabla h_i(Y_i(0)),Y_i'(0)\rangle = -\langle Y_i'(0),Y_i'(0)\rangle$. Also note the relationship $$\|Y_i'(0)\| = \|g_i^{k+1}\|= \|\Proj_{\cT_{x_i^{k+1}}\mathcal{M}_i}\big\{\nabla h_i(Y_i(0))\big\}\| \leq \|\nabla h_i(Y_i(0))\|.$$
	
	Note that
	$\left\|\sum_{i=1}^{N-1}A_i\xke_i+\xke_N-b\right\| \leq \sqrt{\kappa_1\theta_k}\leq \sqrt{\frac{\kappa_1}{\tau}(\Psi_G(x_1^1,...,x_N^1,\lambda^1,x_N^0)-f^*)}.$ Because $\cM_i, i = 1,...,N-1$ are all compact manifolds, $\xke_i, i = 1,...,N-1$ are all bounded. Hence the whole sequence $\{x_N^{k}\}$ is also bounded. By \eqref{To-hard-to-give-a-name-TAT} (which also holds in this case),
	$$\|\lambda^{k+1}\| \leq  |\beta-\frac{1}{\gamma}|\sqrt{\theta_k}+\|\nabla_Nf(\xke_1,\ldots,\xke_{N-1},\xk_N)\|.$$
	By the boundedness of $\{(\xk_1,\ldots,\xk_N)\}$ and the continuity of $\nabla f(\cdot)$, the second term is bounded. Combining the boundedness of $\{\theta_k\}$, we know that whole sequence $\{\lambda^k\}$ is bounded. Consequently, there exists a constant $C>0$ such that $\|\nabla h_i(Y_i(0))\|\leq C,$
	where
	$$\nabla h_i(Y_i(0)) = \nabla_if(x_1^{k+2},...,x^{k+2}_{i-1},\xke_i,...,\xke_N) - A_i^\top\lambda^{k+1}+\beta A_i^\top\bigg(\sum_{j=1}^{i-1}A_jx^{k+2}_j+\sum_{j = i}^N A_j\xke_j - b\bigg)\bigg).$$
	
	Note that this constant $C$ depends only on the first two iterates $\{x_1^t,...,x_N^t,\lambda^t\}, t = 0,1,$ except for the absolute constants such as $\|A_i\|_2,i = 1,...,N$.  Therefore, when
    $$t\leq \frac{2}{2L_2C+\sigma+L_3L_1^2}\leq \frac{2}{2L_2\|\nabla h_i(Y_i(0))\|+\sigma+L_3L_1^2},$$
	it holds that
	$$h_i(Y_i(t))\leq h_i(x^{k+1}_i) - \frac{\sigma}{2}t^2\|g_i^{k+1}\|^2.$$

	Note that $\sigma>\frac{2\alpha}{s}$, by the terminating rule of the line-search step, we have
	$$t_i^k\geq\min\left\{s, \frac{2\alpha}{2L_2C+\sigma+L_3L_1^2}\right\} = \frac{2\alpha}{2L_2C+\sigma+L_3L_1^2}.$$
	Then by noting $$\frac{2\alpha\|g_i^{k+1}\|}{2L_2C+\sigma+L_3L_1^2}\leq t_i^{k+1}\|g_i^{k+1}\|\leq \sqrt{\theta_k},$$  we have \eqref{gradient-bound}.

Now let us discuss the issue of \eqref{opt_ADMM_retraction}. By definition,
	$$g_i^{k+1} = \Proj_{\mathcal{T}_{\xke_i}\mathcal{M}_i}\bigg\{\nabla_if(x_1^{k+2},...,x^{k+2}_{i-1},\xke_i,...,\xke_N) - A_i^\top\lambda^{k+1}+\beta A_i^\top\bigg(\sum_{j=1}^{i-1}A_jx^{k+2}_j+\sum_{j = i}^N A_j\xke_j - b\bigg)\bigg\}.$$
	Consequently, we obtain
	\begin{eqnarray}
	& &   \biggl\|\Proj_{\cT_{x_i^{k+1}}\mathcal{M}_i}\biggl\{\nabla_i f(x^{k+1})-A_i^\top\lambda^{k+1}\biggr\}\biggr\| \nonumber\\
	& = & \biggl\|\Proj_{\cT_{x_i^{k+1}}\mathcal{M}_i}\biggl\{\nabla_i f(x^{k+1})-\nabla_if(x_1^{k+2},\cdots, x_{i-1}^{k+2},x_{i}^{k+1},\cdots,x_{N}^{k+1}) + g_i^{k+1} \nonumber\\
	& & - \beta A_i\left(\sum_{j=1}^NA_jx_j^{k+1}-b\right) + \beta A_i^\T \left(\sum_{j = 1}^{i-1}A_j(x_j^{k+1}-x_j^{k+2})\right)
	\biggr\}\biggr\|  \nonumber\\
	&\leq& \|\nabla_i f(x^{k+1})-\nabla_if(x_1^{k+2},\cdots, x_{i-1}^{k+2},x_{i}^{k+1},\cdots,x_{N}^{k+1})\|+ \|\beta A_i(\sum_{j=1}^NA_jx_j^{k+1}-b)\| \nonumber\\
	& & + \|g_i^{k+1}\| +\|\beta A_i^\T (\sum_{j = i+1}^{N}A_j(x_j^{k+1}-x_j^k) ) \| \nonumber \\
	& \leq & \left(L+\sqrt{N}\beta A_{\max}^2\right)\max\{L_1,1\}\sqrt{\theta_k} + \frac{\sigma+2L_2C+(L+\beta A_{\max}^2)L_1^2}{2\alpha}\sqrt{\theta_{k}}  + \beta\|A_i\|_2 \sqrt{\kappa_1\theta_k} \nonumber\\
	& \leq & \sqrt{\kappa_3\theta_{k}}. \nonumber
	\end{eqnarray}
	
\end{proof}

\subsection{Proof for inequality \eqref{Jacobi}}

\begin{proof}
First, we need to figure out the Lipschitz constant of $\bar{f}_{\beta}$.
\bea
& & \|\nabla\bar{f}_{\beta}(x)-\nabla\bar{f}_{\beta}(y)\| \nonumber\\
&\leq& L\|x-y\| + \beta\left\| \left[ \left(\sum_{j =1}^NA_j(x_j-y_j)\right)^\T A_1,\cdots,\left(\sum_{j =1}^NA_j(x_j-y_j)\right)^\T A_N\right]
\right\| \\
&\leq& L\|x-y\| + \beta\sqrt{N}\max_{1\leq i\leq N}\|A_i\|_2\left\| \sum_{j =1}^NA_j(x_j-y_j)
\right\| \nonumber\\
& \leq &  \left(L+\beta N\max_{1\leq i\leq N}\|A_i\|_2^2 \right)\|x-y\|. \nonumber
\eea
So we define $\hat{L} = L+\beta N\max_{1\leq i\leq N}\|A_i\|_2^2 $ as the Lipschitz constant for
function $\bar{f}_{\beta}.$
The global optimality of the subproblem \eqref{Jacobi-subproblems} yields
\[
\langle\nabla_i\bar{f}_{\beta}(\xk_1,\cdots,\xk_N),\xke_i-\xk_i\rangle-\langle\lambda^k,A_i\xke_i\rangle + r_i(\xke_i)+\frac{1}{2}\|\xke_i-\xk_i\|^2_{H_i}
 \leq  r_i(\xk_i) - \langle\lambda^k,A_i\xk_i\rangle .
\]
By the descent lemma we have
\bea
& & \cL_{\beta}(\xke_1,\cdots,\xke_{N-1},\xk_N,\lambda^k)\nonumber\\
& = & \bar{f}_{\beta}(\xke_1,\cdots,\xke_{N-1},\xk_N) -\langle\lambda^k,\sum_{i=1}^{N}A_i\xke_i-b\rangle +\sum_{i=1}^{N-1}r_i(\xke_i) \nonumber\\
& \leq & \bar{f}_{\beta}(\xk_1,\cdots,\xk_{N-1},\xk_N) +\langle\nabla\bar{f}_{\beta}(\xk_1,\cdots,\xk_{N-1},\xk_N),\xke-\xk\rangle \nonumber\\
& & \frac{\hat{L}}{2}\|x^{k+1}-x^k\|^2-\langle\lambda^k,\sum_{i=1}^{N}A_i\xke_i-b\rangle +\sum_{i=1}^{N-1}r_i(\xke_i). \nonumber
\eea
Combining the above two inequalities yields \eqref{Jacobi}.
\end{proof}

\end{document}